\documentclass[reqno]{amsart}
\usepackage[english]{babel}
\usepackage{amscd,amssymb,amsmath,amsfonts,latexsym,amsthm}
\usepackage{ulem}
\usepackage{comment}
\usepackage{inputenc}
\usepackage{graphicx,color}
\newtheorem{thm}{Theorem}[section]
\newtheorem{cor}[thm]{Corollary}
\newtheorem{lem}[thm]{Lemma}
\newtheorem{prop}[thm]{Proposition}
\newtheorem{defn}[thm]{Definition}
\newtheorem{exam}[thm]{Example}
\newtheorem{rem}[thm]{Remark}

\numberwithin{equation}{section}
\DeclareMathOperator{\Der}{Der}

\def\cal{\mathcal }

\DeclareMathOperator{\codim}{codim}

\DeclareMathOperator{\diag}{diag}

\begin{document}
\title[Pro-nilpotent and pro-solvable Lie algebras]{Infinite dimensional analogues of \\ nilpotent and solvable Lie algebras}

\author{F.H. Haydarov, B.A. Omirov, G.O. Solijanova}

\address{Farhod H. Haydarov \newline \indent
V.I.Romanovskiy Institute of Mathematics,  Uzbekistan Academy of Sciences, Tashkent, Uzbekistan; \newline \indent
National University of Uzbekistan, 100174, Tashkent, Uzbekistan}
\email{{\tt haydarov\_imc@mail.ru}}

\address{Bakhrom A. Omirov \newline \indent
	Institute for Advanced Study in Mathematics, Harbin Institute of Technology, Harbin 150001 \newline \indent
	Suzhou Research Institute, Harbin Institute of Technology, Harbin  215104, Suzhou, China}
\email{{\tt omirovb@mail.ru}}

\address{Gulkhayo O. Solijanova \newline \indent
National University of Uzbekistan, 100174, Tashkent, Uzbekistan\newline \indent
V.I. Romanovskiy Institute of Mathematics, Uzbekistan Academy of Sciences, Tashkent, Uzbekistan
} \email{{\tt gulhayo.solijonova@mail.ru}}

\begin{abstract} We study infinite-dimensional analogues of nilpotent and solvable Lie algebras, focusing on the classes of pro-nilpotent, residually nilpotent, pro-solvable and residually solvable Lie algebras. We extend classical triangularization results (Engel's and Lie's theorems) to the pro-setting and establish existence results for the pro-nilpotent radical in pro-solvable algebras and in certain residually solvable algebras. We adapt finite-dimensional construction methods to produce residually solvable extensions with a given pro-nilpotent radical under natural finiteness conditions.  By analyzing derivations and maximal tori of pro-nilpotent algebras, we extend the notion of rank and show that, for pro-nilpotent algebras of maximal rank, every derivation of a maximal residually solvable extension is inner. Finally, we describe standard constructions (tensor and direct sum products, central extensions) that preserve pro-nilpotency.
\end{abstract}

\subjclass[2020]{17B30, 17B40, 17B65.}

\keywords{pro-nilpotent Lie algebra, pro-solvable Lie algebra, triangularization, residually solvable extension, rank of an algebra, maximal torus, derivation.}

\maketitle

\section*{Introduction}

Infinite-dimensional Lie algebras appear in many parts of mathematics and physics. They arise as algebras of vector fields, as symmetry algebras for differential equations, and in representation theory and conformal field theory (for example, Kac--Moody and Witt-type algebras). Unlike the finite-dimensional case, many classical results do not hold in the infinite-dimensional world.

At the present time there is no general theory of infinite-dimensional Lie algebras. There are, however, four classes of infinite-dimensional Lie algebras that underwent a more or less intensive study due to their various applications mostly in Physics. These are, first of all, the above-mentioned Lie algebras of vector fields, the second class consists of Lie algebras of smooth mappings of a given manifold into a finite-dimensional Lie algebra, the third class is the classical Lie algebras of operators in a Hilbert or Banach space and finally, the fourth class is Kac-Moody algebras.

It is known that analogues of Engel's and Lie's theorems for infinite-dimensional Lie algebras are false. Only in the 1990s E. Zelmanov solved the Burnside problem, which connects the nilpotency property of a Lie algebra with adjoint operators satisfying the $n$-Engel's condition \cite{Zelmanov}.

The study of narrow Lie algebras was initiated by Shalev and Zelmanov~\cite{Shalev,Shalev1}, who investigated positively graded Lie algebras characterized by a small width parameter~$d$. Formally, a positively graded Lie algebra $\mathfrak{g} = \bigoplus\limits_{i=1}^{\infty} \mathfrak{g}_i$ is said to be {\it finite width}, if the dimensions of all its homogeneous components $\mathfrak{g}_i$ are uniformly bounded for all $i$. The authors underscored the profound difficulty inherent in classifying graded Lie algebras of finite width, remarking that it presents a ``formidable challenge''~\cite{Shalev1}. The classification of naturally graded Lie algebras of width greater than one is a formidable challenge, one that lies beyond the reach of current methods. A crucial and more accessible subclass within this family consists of those algebras satisfying a specific narrowness criterion:
$\dim \mathfrak{g}_i + \dim \mathfrak{g}_{i+1} \leq 3$ for all $i \geq 1$  \cite{MillionshSMJ}.

Finite-dimensional Lie algebras satisfying the conditions $\dim \mathfrak{g}_{1} = 2$ and $\dim \mathfrak{g}_{i} = 1$ for all $i \geq 2$, known as filiform Lie algebras, were introduced by Vergne \cite{Vergne}. Their infinite-dimensional analogue, referred to as $\mathbb{N}$-graded Lie algebras of maximal class, was subsequently investigated by Fialowski \cite{Fialowski} and \cite{Mil2008}. Similarly, an infinite-dimensional Leibniz analogue of filiform Lie algebras was later proposed in \cite{thin}. These algebras are a particular case of the pro-nilpotent algebras. The study of pro-nilpotent and pro-solvable Lie algebras and their relation to Klein-Gordon PDE is the focus of \cite{MillionshART}. It shows that the applications related to infinite-dimensional analogues of naturally graded nilpotent Lie algebras are of special interest. For instance, the characteristic Lie algebras of some hyperbolic partial differential equations are naturally graded and if Darboux integrability of such equations implies the finite-dimensionality of characteristic Lie algebras, then integrability in the sense of the inverse scattering problem method leads to slowly growing infinite-dimensional Lie algebras. The latest progress in this direction has been obtained by Millionshchikov \cite{MillionshART}, who studied the characteristic Lie algebras of some partial hyperbolic equations and their connection with narrow naturally graded Lie algebras.


The description of pro-nilpotent Lie algebras differs from the usual
classifications since there are no simple objects, in the sense that each positively graded Lie algebra has a proper ideal. From this point of view, it will differ, for example, from the classification of simple $\mathbb{Z}$-graded Lie algebras obtained in \cite{Kac,Kac1,Mathieu}. It is also known that every pro-nilpotent Lie algebra is a special filtered deformation of some Carnot algebra.

The study of finite-dimensional solvable Lie algebras with special types of nilradicals comes from different problems in Physics (see in \cite{Snobl1} and reference therein). By analogy with solvable Lie algebras, we consider residually solvable and pro-solvable algebras, as algebras defined defined by two similar properties: the intersection of all members of the derived series is equal to zero (for residually solvable property), and the quotient algebra by any member of the derived series is finite-dimensional. A typical example of a residually solvable Lie algebra is the non-negative part of the Witt algebra whose maximal pro-nilpotent ideal is its positive part. In this case the method of constructing a solvable finite-dimensional Lie algebra by means of its nilradical, given in \cite{Mubor}, is completely consistent for these infinite-dimensional Lie algebras. The method has also been successfully implemented for other examples, including pro-solvable Lie algebras whose maximal pro-nilpotent ideals are $\mathbb{N}$-graded Lie algebras of maximal class and the positive parts of the affine Kac-Moody algebras $A_1^{(1)}, A_2^{(2)}$ (see in \cite{Residual1,Residual2})

This paper we focus our study on infinite-dimensional analogues of nilpotent and solvable Lie algebras, so-called pro-nilpotent and pro-solvable algebras, considered by D.V. Millionshchikov in \cite{Millionshdis}, \cite{MillionshART}. For these algebras, analogues of Engel's and Lie's theorems in terms of strictly triangunalizability and triangularizability are obtained.

Defining a maximal torus in a pro-nilpotent Lie algebra $\mathcal{N}$ requires special care due to infinite-dimensional phenomena. We prove that a maximal torus in $\mathcal{N}$ can be obtained as the limit of maximal tori in the finite-dimensional nilpotent quotients $\mathcal{N} / \mathcal{N}^k$. Furthermore, we show that any two maximal tori in a pro-nilpotent Lie algebra have the same dimension. This result ensures that the notion of rank is well-defined for pro-nilpotent algebras.

We also introduce the concept of residually solvable algebras and show that, under certain conditions, such algebras contain a maximal (with respect to inclusion) pro-nilpotent ideal. Lie algebras of this type are referred to as residually solvable extensions. By examining further properties of residually solvable extensions, we demonstrate that the classical method of constructing finite-dimensional solvable Lie algebras by means their nilradicals and certain types of derivations can be naturally extended to this class of residually solvable algebras with a prescribed pro-nilpotent radical. Moreover, we prove that for a maximal residually solvable extension of a pro-nilpotent Lie algebra of maximal rank with an abelian complementary subalgebra, the resulting extension is centerless and all its derivations are inner.  Finally, we present examples of constructions that generate new ``pro-algebras'' from existing ones, in some sense, closed under ``pro-properties''.

The paper is organized as follows: In Section 1 we introduce pro-nilpotent (residually nilpotent) and pro-solvable (residually solvable) Lie algebras. It is established their elementary properties, proves the closure under finite sums of ideals and the filtration topology is introduced.

Section 2 develops infinite-dimensional analogues of classical triangularization results. We extend Engel's and Lie's theorems to pro-nilpotent and pro-solvable Lie algebras, showing that their adjoint representations admit triangularizable (or strictly triangularizable) forms. In particular, a pro-nilpotent algebra has strictly triangularizable adjoint operators, while a pro-solvable algebra has triangularizable ones. We also establish structural consequences, such as the pro-nilpotency of the derived square, and describe the behavior of derivations in terms of block-triangular forms. These results highlight similarities with the finite-dimensional setting and at the same time some differences are also remarked.

In Section 3 we extend the results on maximal tori on pro-nilpotent Lie algebras. We prove that a maximal torus on a pro-nilpotent Lie algebra $\cal N$ has dimension bounded above by number of generators of $\cal N$ and dimensions of two maximal tori are coincide, which allows us to define the rank of $\cal N$. We describe maximal tori as limits of compatible systems of tori on finite-dimensional nilpotent quotients and establish root space decompositions with respect to such tori. Finally, for complete pro-nilpotent Lie algebras, we show that the exponential of a derivation converges to an automorphism.

Section 4 investigates residually solvable extensions of pro-nilpotent Lie algebras, extending the classical theory of solvable extensions in finite dimensions. We introduce the notion of the pro-nilpotent radical and prove its existence under natural restrictions in residually solvable algebras. We extend the classical construction of finite-dimensional solvable Lie algebras via their nilradicals to residually solvable extensions with a given pro-nilpotent radical. Moreover, it is established that in a maximal residually solvable extension of a pro-nilpotent algebra of maximal rank, with an abelian complementary subalgebra, all derivations are inner and the algebra is centerless. The final Section we show that some known constructions provide new pro-nilpotent (pro-solvable) Lie algebras from given ones.

Throughout the paper we shall consider only countable dimensional complex algebras and vector spaces with Hamel bases.

\section{On pro-nilpotent and pro-solvable structures}

In this section we give the concepts of pro-nilpotent (residually nilpotent), pro-solvable (residually solvable) Lie algebras and provide several examples of such algebras. Some basic results for these type of algebras are given. Moreover, the filtration topology on pro-nilpotent Lie algebra is presented as well.

For a given Lie algebra $\cal L$ we define the lower central  and derived series as follows:
$$\cal L^{k+1}=[\cal L^{k},\cal L],  \quad \cal L^{[k+1]}=[\cal L^{[k]},\cal L^{[k]}], \ k\geq 1.$$

In \cite{MillionshSMJ}, D.V. Millionschikov considered special types of infinite-dimensional Lie algebras for which the dimensions of the terms in the lower central series are strictly decreasing and never become zero. In the finite-dimensional case, a strictly decreasing dimension of the terms in the lower central series implies the nilpotency of the algebra. Applying a similar approach, we introduce the concept of a pro-solvable Lie algebra.

\begin{defn}\label{def2} Let $\cal L$ be a infinite-dimensional Lie algebra.

(i)  If $\cal L$ satisfies $\bigcap\limits_{i=1} ^{\infty}\cal L^{i}=0$ (resp. $\bigcap\limits_{i=1} ^{\infty}\cal L^{[i]}=0$) then it is called residually nilpotent (resp. residually solvable);

(ii) If residually nilpotent (resp. residually solvable) algebra satisfies $\dim (\cal L /\cal L^{i+1}) < \infty$ (resp. $\dim (\cal L /\cal L^{[i+1]}) < \infty$) for any $i\geq 1$ then it is called pro-nilpotent (resp. pro-solvable).
\end{defn}

In paper \cite{Yu}, two classes of infinite dimensional Lie algebras are considered. Namely, a residually nilpotent (resp. residually solvable) algebra $\cal L$ is called potentially nilpotent (respectively, potentially solvable), if
$\dim (\cal L^i /\cal L^{i+1}) < \infty$ (resp. $\dim (\cal L^{[i]} / \cal L^{[i+1]}) < \infty$) for any $i\geq 1$.

The vector space isomorphism
$$\cal L / \cal L^{i} \cong \cal L / \cal L^{2} \oplus \cal L^2 / \cal L^{3}\oplus \cdots \oplus \cal L^{i} / \cal L^{i+1},$$
implies that
$\dim(\cal L / \cal L^{i})=\sum\limits_{j=1}^{i} \dim (\cal L^j / \cal L^{j+1})< \infty.$
Therefore,
\begin{equation} \label{eq1}
\dim(\cal L / \cal L^{i}) < \infty \quad \forall \ i\in \mathbb{N} \quad  \Leftrightarrow \quad
\dim(\cal L^{i} / \cal L^{i+1}) < \infty \quad  \forall \  i\in \mathbb{N}.
\end{equation}

Thus, definitions of pro-nilpotency (resp. pro-solvability) and potentially nilpotency (resp. potentially solvability) are equivalent.

Below we present examples of infinite-dimensional Lie algebras considered in \cite{Fialowski} and \cite{MillionshART}, which satisfy some properties related to notions from Definition \ref{def2}.

\begin{exam}\label{examples}
\

\begin{itemize}
\item[--] The non-negative part of Witt algebra is pro-solvable and non residually nilpotent algebra (consequently, non pro-nilpotent), while positive part of Witt algebra is both pro-nilpotent and pro-solvable;

\item[--] The solvable Lie algebra $\mathfrak{a}_{\infty}$ with given on the basis $\{e_i \ | \ i\in \mathbb{Z}_{+}\}$ by non zero products:
$$[e_0, e_i] = e_{i-1}, \ i \geq 3,$$
is neither pro-solvable nor residually nilpotent;

\item[--] The following infinite-dimensional filiform Lie algebras given on the basis
$\{e_i \ | \ i\in \mathbb{N}\}$ with multiplications tables:
$$\begin{array}{llllllll}
\rm m_1: & [e_1,e_i]=e_{i+1}, &  i \geq 2; &\\[3mm]
\rm m_2: & [e_i,e_1]=e_{i+1}, &i \geq 2, & [e_2,e_j]=e_{j+2}, & j \geq 3; \\[3mm]\end{array}
$$
are pro-nilpotent and residually solvable, but non pro-solvable;
\end{itemize}
\end{exam}

We extend the notions of Definition~\ref{def2} to subalgebras and ideals of an algebra $\mathcal{L}$; a subalgebra or ideal is called {\it residually nilpotent, pro-nilpotent, etc.}, if it satisfies the corresponding conditions of Definition~\ref{def2}.

We have the following observations:
\begin{itemize}
\item[-] $Span \{e_{2k} \ | \ k\in \mathbb{N}\}$ is a residually nilpotent subalgebra (not ideal) of $\rm m_1$;

\item[-] $Span\{e_i \ | \ i\geq 2\}$ forms a subalgebra of $\rm m_1$ which is not pro-nilpotent;

\item[-] $Span\{e_1, e_i \ | \ i\geq 3\}$ is both pro-nilpotent subalgebra and pro-nilpotent ideal of $\rm m_1$.
\end{itemize}

\begin{rem}
Therefore, subalgebra (ideal) of pro-nilpotent Lie algebra is not pro-nilpotent algebra, in general. Similarly is true for pro-solvable Lie algebras. \end{rem}

Let us consider another special family of algebras.
\begin{exam} For a given $s\in \mathbb{N}\cup \{0\}$ we consider the pro-nilpotent Lie algebra
$$\cal W(s): [e_i,e_j]=(j-i)e_{i+j+s},\ i,j\geq1.$$
Note that $\cal W(0)$ is the positive part of Witt algebra and it is a $2$-cocycle of $\cal W(s)$. Since $\cal W(s)$ is a $2$-cocycle of $\cal W(0)$, we deduce that the family of pro-nilpotent Lie algebras
$$
[e_i,e_j]=(j-i)\sum\limits_{k=0}^{t}a_ke_{i+j+k},  \quad i,j\geq 1 \quad a_k\in\mathbb C.
$$
consists of infinitesimal deformations of the algebra $\cal W(0)$. It is remarkable that algebra $\cal W(s)$ is a mutant of the algebra $\cal W(0)$ realized by its $\frac{1}{2}$-derivation defined as $d(e_i)=e_{i+s}, \ i \geq 1$ (see for details on mutant and $\delta$-derivations in \cite{Fil1}).
  \end{exam}

The following results will be used frequently throughout this paper.
\begin{lem} [Closure of ideal sum] \label{lem:pro-nilpotent-sum}
\

(i) The sum of two pro-nilpotent ideals of a Lie algebra is pro-nilpotent.

(ii) The sum of two pro-solvable ideals of a Lie algebra is pro-solvable.
\end{lem}

\begin{proof} (i) Let $\cal I$ and $\cal J$ be pro-nilpotent ideals of a Lie algebra $\mathcal{R}$. Clearly, $\cal I + \cal J$ is an ideal of $\cal R$. First we show that $\cal I + \cal J$ is a residually nilpotent ideal of $\cal R$. Since $(\cal I + \cal J)^{2i-1}\subseteq \cal I^i + \cal J^i$ it suffices to show that $\bigcap\limits_{i=1}^{\infty} (\cal I^i + \cal J^i)=\{0\}$.

For any natural $i\in \mathbb{N}$ any element of the space $\cal I^i + \cal J^i$ is decomposed into sum of elements of the following sets:
$\cal I^{i}\setminus \cal J^{i}, \quad \cal J^{i}\setminus \cal I^{i}, \quad \cal I^{i}\cap \cal J^{i}.$ Taking into account the followings
\begin{equation}\label{eq2.3}
\bigcap_{i=1}^{\infty}(\cal I^i\setminus \cal J^i)\subseteq \bigcap_{i=1}^{\infty}\cal I^i=0, \quad \bigcap_{i=1}^{\infty}(\cal J^i\setminus \cal I^i)\subseteq \bigcap_{i=1}^{\infty}\cal J^i=0, \quad \bigcap_{i=1}^{\infty}(\cal I^i\cap \cal J^i)\subseteq \bigcap_{i=1}^{\infty}\cal I^i=0,\end{equation}
we conclude
$\bigcap\limits_{i=1}^{\infty} (\cal I^i + \cal J^i)=\{0\}$.

Due to embedding $\cal I^i + \cal J^i \subseteq (\cal I + \cal J)^{i}$ for any $i$, we obtain
$$
\dim \frac{\cal I + \cal J}{(\cal I + \cal J)^i} \leq \dim \frac{\cal I + \cal J}{\cal I^i + \cal J^i}\leq \dim \frac{\cal I}{\cal I^i}+\dim \frac{\cal J}{\cal J^i}.
$$
Since for any $i$ we have $\dim (\cal I/ \cal I^i) < \infty$ and $\dim (\cal J/\cal J^i) < \infty$, then $\dim \frac{\cal I + \cal J}{(\cal I + \cal J)^i}< \infty$.

(ii) Let $\cal I$ and $\cal J$ be pro-solvable ideals of a Lie algebra $\mathcal{R}$. By induction we prove the following embedding:
$$(\cal I + \cal J)^{[2i]}\subseteq \cal I^{[i]} + \cal J^{[i]}.$$ The base of induction is obvious. Applying the equality $(\cal I + \cal J)^{[s+t-1]}=\big((\cal I + \cal J)^{[s]}\big)^{[t]}$ for any $s, t\in \mathbb{N}$ and induction assumption we derive

$$\begin{array}{lccllllll}
(\cal I + \cal J)^{[2i+2]}&=&
\big((\cal I + \cal J)^{[2i+1]}\big)^{[2]}
&\subseteq &\big[\cal I^{[i]} + \cal J^{[ i]}, \cal I^{[i]} + \cal J^{[i]}\big]^{[2]}\\[3mm]
&\subseteq &\left(\cal I^{[i+1]} + \cal J^{[i+1]}+\cal I^{[i]} \cdot \cal J^{[i]}\right)^{[2]}&\subseteq &
\cal I^{[i+1]} + \cal J^{[i+1]}.
\end{array}$$

Now applying the same argumants as in \eqref{eq2.3} we derive that we get $\bigcap\limits_{i=1}^{\infty} (\cal I + \cal J)^{[i]}=0.$

Finally, the following chain of inequalities
$$
\dim \frac{\cal I + \cal J}{(\cal I + \cal J)^{[i]}} \leq  \dim \frac{\cal I + \cal J}{(\cal I + \cal J)^{[2i]}} \leq
\dim \frac{\cal I + \cal J}{\cal I^{[i]} + \cal J^{[i]}}\leq \dim \frac{\cal I}{\cal I^{[i]}}+\dim \frac{\cal J}{\cal J^{[i]}}<\infty
$$
complete the proof that $\cal I + \cal J$ is pro-solvable.
\end{proof}

From Lemma \ref{lem:pro-nilpotent-sum} we conclude that the sum of finite many pro-nilpotent (pro-solvable) ideals of a Lie algebra is again pro-nilpotent (pro-solvable) ideal. Below we present examples which shows that the sum of infinitely many ideals of pro-nilpotent (pro-solvable) Lie algebras is not pro-nilpotent (pro-solvable) anymore.
\begin{exam} \label{exam2.6} Consider a Lie algebra $\cal L=\bigoplus\limits_{n\in\mathbb{N}}I_n$, which is a direct sum of pro-nilpotent (pro-solvable) ideals $I_i, i\in\mathbb{N}$ of $\cal L$. Then $\cal L$ is non pro-nilpotent (pro-solvable) algebra. Therefore, the sum of infinitely many pro-nilpotent (pro-solvable) ideals no need to be pro-nilpotent (pro-solvable).
\end{exam}

\begin{prop} \label{prop1.3} The following statements hold:

(i) Any subalgebra of a residually nilpotent (resp. residually solvable) algebra is residually nilpotent (resp. residually solvable),

(ii) Residually nilpotent Lie algebra is residually solvable,

(iii) Let $\cal L$ be a Lie algebra and $\cal I$ be its residually nilpotent ideal such that $\cal L^2 \subseteq \cal I.$ Then $\cal L$ is residually solvable.

($iv$) Let $\cal R$ be a pro-solvable Lie algebra. Then $\cal R^{[i]}$ is pro-solvable for any $i\in \mathbb{N}.$

($v$) Let $\cal I$ be a residually solvable ideal of a Lie algebra $\cal L$ such that the quotient algebra $\cal L / \cal I$ is abelian, then $\cal L$ itself is residually solvable.
\end{prop}
\begin{proof} Let $\cal M$ be a subalgebra of
residually nilpotent (resp. residually solvable) algebra $\cal L$. Then the part ($i$) immediately  follows from embeddings: $\mathcal{M}^i \subseteq \mathcal{L}^i, \
\mathcal{M}^{[i]} \subseteq \mathcal{L}^{[i]}.$

From $ \mathcal{L}^{[i]} \subseteq \mathcal{L}^{2^{i-1}},$ we get $\bigcap\limits_{i=1}^{\infty} \mathcal{L}^{[i]}\subseteq
\bigcap\limits_{i=1}^{\infty} \mathcal{L}^{2^{i-1}}=
\bigcap\limits_{i=1}^{\infty}\mathcal{L}^{i}=0,$ which proves the part ($ii$).

($iii$) The condition $\cal L^2 \subseteq \cal I$ implies that $\cal L^{[i+1]} \subseteq \cal I^{[i]}.$ In addition, we have $\cal I^{[i]} \subseteq \cal I^{2^{i-1}}.$ Now, residually nilpotency of $\cal I$ implies residually solvability of $\cal L.$

($iv$) Set $\cal M=\mathcal{R}^{[i]}$. The equality $\mathcal{M}^{[j]} = \mathcal{R}^{[i+j-1]}$ for any $i, j \ge 1$, implies
$$\bigcap\limits_{j=1}^{\infty}\cal M^{[j]}=
\bigcap\limits_{j=1}^{\infty} \cal R^{[i+j-1]}\subseteq \bigcap\limits_{j=1}^{\infty} \cal R^{[j]}=0.$$
Moreover, $\dim(\cal R^{[j]} / \cal R^{[j+1]}) < \infty$ guarantees that $\dim(\cal M^{[j]} / \cal M^{[j+1]}) < \infty.$

The proof of part ($v$) is obvious.
\end{proof}

Let $\mathcal{L}$ be a pro-nilpotent Lie algebra with descending chain of ideals
$$\mathcal{L} = \mathcal{L}^1 \supset \mathcal{L}^2 \supset \mathcal{L}^3 \supset \cdots$$
The \textit{filtration topology} on $\mathcal{L}$ is defined by declaring the collection $\{\mathcal{L}^k  \ | \ k \in \mathbb{N}\}$ to be a neighborhood basis of the zero element. Equivalently, a subset $U \subseteq \mathcal{L}$ is open if and only if for every $x \in U$ there exists $k \ge 1$ such that $x + \mathcal{L}^k \subseteq U$. By translation invariance, the sets $\{x + \mathcal{L}^k : k \ge 1\}$
form a fundamental system of neighborhoods about an arbitrary $x \in \mathcal{L}$. Moreover, since
$\bigcap\limits_{k=1}^\infty \mathcal{L}^k = \{0\},$
the filtration topology is Hausdorff. The Lie algebra $\mathcal{L}$, equipped with the filtration topology, is said to be \textit{complete} if every Cauchy sequence in $\mathcal{L}$ converges to some limit in $\mathcal{L}$. Note that there is an inverse spectrum of finite-dimensional nilpotent Lie algebras
$$\cdot \cdot \cdot \xrightarrow{p_{k+2,k+1}} \cal L/ \cal L^{k+1} \xrightarrow{p_{k+1,k}}\cal L/\cal L^{k}
\xrightarrow{{p}_{k,k-1}}\cdot \cdot \cdot  \xrightarrow{p_{3,2}}\cal L/\cal L^2 \xrightarrow{p_{2,1}}\cal L /\cal L=\{0\}.
$$
Equivalently, completeness can be seen via the inverse-limit construction
$\mathcal{L} \cong \lim\limits_{\xleftarrow{k}} (\cal L /\cal L^k),$ where each quotient $\mathcal{L}/\mathcal{L}^k$ is finite-dimensional (see e.g. \cite{MillionshSMJ}).

Let $f:\mathcal{L}\to\mathcal{L}$ be a linear transformation of a pro-nilpotent Lie algebra endowed with filtration topology on it. Then $f$ is \textit{continuous} if and only if for every $k$ there exists $m(k)$ such that
$f\bigl(\mathcal{L}^{m(k)}\bigr) \subseteq \mathcal{L}^k$, or we can say $f$ is continuous if and only if it preserves the filtration, i.e., $f\bigl(\mathcal{L}^k\bigr) \subseteq \mathcal{L}^k \quad\text{for all }k\in \mathbb{N}.$ Taking into account that $\mathcal{L}^k$ are characteristic ideals, one can conclude that any derivation of a pro-nilpotent Lie algebra is continuous.

Note that for a pro-solvable Lie algebra $\mathcal{L}$
the filtration topology is defined by taking $\{\mathcal{L}^{[k]}:k\ge0\}$ as a neighborhood basis of $0$. Then one can introduce all concepts related to the topology similar to the case of pro-nilpotent.

\section{Triangularizability of pro-nilpotent and pro-solvable Lie algebras.}

This section extends fundamental triangularization theorems from finite-dimensional Lie theory to pro-nilpotent and pro-solvable algebras. We establish analogues of Engel's and Lie's theorems, characterizing when adjoint representations become (strictly) triangularizable. We also prove a consequence of Lie's theorem and present result on derivations of pro-nilpotent algebras.

\begin{defn} A family $\mathcal{T}$ of linear transformations on the linear Lie algebra $\mathfrak{g l}(\mathcal{L})$ is triangularizable if there is a chain $\mathcal{C}$ that is maximal as a chain of subspaces of $\mathcal{L}$ and that has the property that every subspace in $\mathcal{C}$ is invariant under all the transformations in $\mathcal{T}$. Any such chain of subspaces is said to be a triangularizing chain for the family $\mathcal{T}$.
\end{defn}

We require an infinite-dimensional triangularization lemma analogous to the finite-dimensional one.

Recall, that Zorn's lemma states that a partially ordered set containing upper bounds for every chain (that is, every totally ordered subset) necessarily contains at least one maximal element.

\begin{defn}\label{defn2.8}
\

(i)  A linear transformation  $T$ of a vector space $\mathcal{L}$ is called triangularizable, if there exists a chain of $T$-invariant subspaces of $\mathcal{L}$ :
\begin{equation}\label{eeee}
\mathcal{L}=\mathcal{V}_1 \supset \mathcal{V}_2 \supset \cdots \supset \mathcal{V}_i \supset \mathcal{V}_{i+1} \supset \cdots \supset \{0\}, \quad \dim\big(\cal V_{i}/ \cal V_{i+1}\big)=1.
\end{equation}
Furthermore, it is called strictly triangularizable if $T\left(\mathcal{V}_i\right) \subseteq \mathcal{V}_{i+1}$ for all $i \in \mathbb{N}$;

(ii) $T$ is called pro-nilpotent, if $\bigcap\limits_{i=1} ^{\infty}Im(T^i)=\{0\}$ and $\operatorname{dim}\left(T^i(\mathcal{L}) / T^{i+1}(\mathcal{L})\right)<\infty$ with $T^0(\mathcal{L})=\mathcal{L}, \ i \geq 0.$

(iii) The set $ad(\cal L)$ is called triangularizable (respectively, strictly triangularizable), if there is a chain of ideals
(\ref{eeee})
such that $ad(\cal L)(\cal V_{i}) \subseteq \cal V_{i}$
(respectively, $ad(\cal L)(\cal V_{i}) \subseteq \cal V_{i+1}$) for all $i\in \mathbb{N}$;

\end{defn}

Note that Definition \ref{defn2.8} was considered in \cite{Radjavi} for the study of bounded linear operators in Banach spaces. For linear transformations of a vector space, in above we presented the notion of pro-nilpotency, which is weaker than the triangularizable operator. In other words, every triangularizable operator is pro-nilpotent one, while the contrary is not true, in general.

\subsection{Analogue of Lie's theorem.}

\begin{prop}\label{prop1.7} Let $\mathcal{L}$ be a Lie algebra such that the family $ad(\mathcal{L})$ is strictly triangularizable. Then $\mathcal{L}$ is either nilpotent or residually nilpotent.
\end{prop}

\begin{proof} Let assume that the lower central series reaches zero after finitely many steps, then from strictly triangularizability one gets nilpotency of $\cal L$.

Consider now the case when $\cal L^{k+1}\subsetneq \cal L^k$ and $\cal L^k\neq 0$ for any $k\in \mathbb{N}$. Let $\mathfrak{C}$ be a maximal chain of ideals $\cal V_i$ such that
$$\cal L =\cal V_1 \supset \cal V_2 \supset \dots \supset \cal V_i \supset \cal V_{i+1} \supset \dots \supset \{0\}$$
with $dim\big(\cal V_{i}/ \cal V_{i+1}\big)=1$ and $ad(\cal L)(\cal V_i) \subseteq \cal V_{i+1} $ for all $i\in \mathbb{N}$. Let us assume that $\bigcap\limits_{i=1} ^{\infty}\cal L^{i}\neq 0$. Consider $0\neq x\in \bigcap\limits_{i=1} ^{\infty}\cal L^{i}$. Since $\cal L=\cal V_1$, then applying $ad(\cal L)(\cal V_i) \subseteq \cal V_{i+1}$ we derive $\cal L^{s}=ad(\cal L)^{s-1}(\cal L) \subseteq \cal V_s$ for any $s\in \mathbb{N}.$ It implies $x\in \bigcap\limits_{i=1} ^{\infty}\cal V_s=0,$ which contradicts to the choice of $x$. Thereby, $\bigcap\limits_{i=1} ^{\infty}\cal L^{i}=0$.

If this the lower central series reaches the zero algebra after finitely many steps (i.e., $\mathcal{L}^k = 0$ for some $k \in \mathbb{N}$), then $\mathcal{L}$ is a nilpotent Lie algebra; otherwise, it is residually nilpotent.
\end{proof}

From the above assertion naturally question arises:
{\it Is an infinite-dimensional Lie algebra with a strictly triangularizable adjoint representation necessarily nilpotent or pro-nilpotent?}
The next example shows a Lie algebra satisfying the conditions is not necessarily pro-nilpotent.

\begin{exam} Let $\cal L$ be the direct sum of a finite-dimensional nilpotent Lie algebra of nilindex $s$ and infinite dimensional abelian Lie algebra. Then $ad(\mathcal{L})$ is strictly triangularizable, while it is not pro-nilpotent Lie algebra (because of $dim (\cal L/ \cal L^2)=\infty).$
\end{exam}

\begin{prop}\label{Engelpart1} (Strictly Triangularization) Let $\mathcal{L}$ be an infinite-dimensional Lie algebra. Then the following statements hold:

(i) If $\mathcal{L}$ is pro-nilpotent, then the family of operators $ad(\mathcal{L})$ is strictly triangularizable;

(ii) If an element $x$ of a Lie algebra $\mathcal{L}$ the operator $\operatorname{ad}_x$  is pro-nilpotent, then it is strictly triangularizable.

\end{prop}
\begin{proof} $(i)$ Let us consider the set $\{{\cal C}_{\alpha}\}_{\alpha\in \cal I}$ of totally ordered ideals (by inclusion) of $\cal L$. With the set $\{{\cal C}_{\alpha}\}_{\alpha\in \cal I}$ one can correspond the chain (denoted by $\mathfrak{C}$). We denote by $\cal P$ the sets of all chains and on $\cal P$
we consider the following partially order:
$$\mathfrak{C}_{1}\leq \mathfrak{C}_{2}  \Leftrightarrow
\{{\cal C}_{\alpha}\}_{\alpha\in \cal I_1}
\subseteq
\{{\cal C}_{\beta}\}_{\beta\in \cal I_2}.$$

It is clear that trivial chains (that is, $\{0\}$ and $\{0, \cal L\}$) lie in $\mathcal{P}$. For a totally ordered family of chains $\{\mathfrak{C}_{\alpha}\}_{\alpha\in \mathfrak{I}}\subseteq \mathcal{P}$  the following set $\bigcup\limits_{\alpha \in \mathfrak{I}} \mathfrak{C}_{\alpha}$ is an element of $\cal P$.

Indeed, if $M$ and $N$ are ideals of the family $\bigcup\limits_{\alpha \in \mathfrak{I}} \mathfrak{C}_{\alpha},$ then $M\in \mathfrak{C}_{\alpha_1}$ and $N\in \mathfrak{C}_{\alpha_2}$ for some $\alpha_1, \alpha_2\in \mathfrak{I}.$
Since the family $\{\mathfrak{C}_{\alpha}\}_{\alpha\in \mathfrak{I}}$ is totally ordered, without loss of generality, one can assume that
$\mathfrak{C}_{\alpha_1}\leq \mathfrak{C}_{\alpha_2}$, which implies that $M, N \in \mathfrak{C}_{\alpha_2}$. Taking into account that $\mathfrak{C}_{\alpha_2}$ is a chain, we deduce that $M$ and $N$ are comparable. Consequently, $\bigcup\limits_{\alpha \in \mathfrak{I}} \mathfrak{C}_{\alpha}\in \cal P$. Thus, every totally ordered family of  elements of $\mathcal{P}$ has an upper bound in $\mathcal{P}$. Zorn's Lemma now guarantees that $\mathcal{P}$ has at least one maximal chain. A maximal chain refers to a chain that a maximal element of the partially ordered set $\cal P$ with respect to inclusion.

We denote this maximal chain by $\cal M$.
For an ideal $\mathcal{Q}$ of the chain $\mathcal{M}$ we set $\mathcal{Q}_{-}$ is defined as follows
$$
\mathcal{Q}_{-}=\bigvee\{\mathcal{N} \in \mathcal{M}: \mathcal{N} \subsetneq \mathcal{Q}\},
$$
where $\vee$ is supremum in partially ordered set (poset).

Since $\mathcal{L}$ is pro-nilpotent Lie algebra. Then, without loss of generality, we can consider that a chain $\mathcal{M}$ contain the family of ideals $\{\{0\}, \{\cal L^i\}_{i\in\mathbb{N}}\}$.

Suppose that for some $\mathcal{Q} \in \mathcal{M}$ we have $dim(\mathcal{Q} / \mathcal{Q}_{-})>1.$ From the condition $\bigcap\limits_{i=1} ^{\infty}\cal L^{i}=\{0\}$
we get the existence $s\in \mathbb{N}$ such that $\mathcal{L}^{s+1}\subseteq \mathcal{Q}_{-}\subset \mathcal{Q}\subseteq\mathcal{L}^s.$ Note that pro-nilpotency of $\cal L$ implies that dimension of $\mathcal{Q} / \mathcal{Q}_{-}$ is finite. Therefore, for an element $e\in \cal Q\setminus \mathcal{Q}_{-}$ the space $\mathcal{Q}_{-}(e)=\langle e \rangle \oplus \mathcal{Q}_{-}$ forms an ideal of $\cal L$. It follows that $\mathcal{Q}_{-}\subsetneq \mathcal{Q}_{-}(e) \subsetneq \mathcal{Q}$, which contradicts to maximality of $\cal M$.

Taking into account that $$Ad(\cal L)(\cal Q)\subseteq Ad(\cal L)(\cal L^s)=\cal L^{s+1} \subseteq \mathcal{Q}_{-}$$
we conclude that $Ad(\cal L)$ is strictly triangularizable with respect to the chain $\cal M$.

$(ii)$  Let $\mathcal{B} = \{e_1, e_2, \ldots\}$ be a basis of $\mathcal{L}$ that is compatible with the descending filtration given by the images of the iterated adjoint maps:
\begin{equation} \label{chain}
  \mathcal{L} = \operatorname{ad}_x^0(\mathcal{L}) \supset \operatorname{ad}_x(\mathcal{L}) \supset \cdots \supset \operatorname{ad}_x^i(\mathcal{L}) \supset \operatorname{ad}_x^{i+1}(\mathcal{L}) \supset \cdots \supset 0.
\end{equation}

Following the approach used in the part $(i)$, we consider a maximal chain $\mathcal{M}$ of subspaces of $\mathcal{L}$ that contains the chain in~\eqref{chain}. Suppose, for the sake of contradiction, that there exists a subspace $\mathcal{Q} \in \mathcal{M}$ such that $\dim(\mathcal{Q} / \mathcal{Q}_{-}) > 1$, where $\mathcal{Q}_{-}$ denotes the immediate predecessor of $\mathcal{Q}$ in the chain. Since $\operatorname{ad}_x$ is pro-nilpotent, there exists an integer $s \in \mathbb{N}$ such that
\[
ad_x^{s+1}(\mathcal{L}) \subseteq \mathcal{Q}_{-} \subset \mathcal{Q} \subseteq ad_x^{s}(\mathcal{L}),
\]
which implies $\dim(\mathcal{Q} / \mathcal{Q}_{-}) < \infty$. Let $e \in \mathcal{Q} \setminus \mathcal{Q}_{-}$. Consider the subspace $\mathcal{Q}_{-}(e) := \langle e \rangle \oplus \mathcal{Q}_{-}$. This subspace is invariant under the action of $\operatorname{ad}_x$ and $$ad_x(\mathcal{Q}_{-}(e))\subseteq ad_x(ad_x^s(\cal L))=ad_x^{s+1}(\mathcal{L})\subseteq \mathcal{Q}_{-}.$$

Therefore, we obtain a refinement of the chain:
$\mathcal{Q}_{-} \subsetneq \mathcal{Q}_{-}(e) \subsetneq \mathcal{Q},$
which contradicts the maximality of the chain $\mathcal{M}$. Hence, no such subspace $\mathcal{Q}$ can exist, and the operator $\operatorname{ad}_x$ must be strictly triangularizable.
\end{proof}

The following result represents  an analogue of Lie's theorem for pro-solvable Lie algebras.
\begin{thm} \label{prop1.10} (Triangularization) Let $\mathcal{L}$ be a pro-solvable, then the family $ad(\mathcal{L})$ is triangularizable.
\end{thm}
\begin{proof} Applying the arguments similar as in the proof of Proposition \ref{Engelpart1} we derive that for pro-solvable Lie algebra $\mathcal{L}$ there exists a maximal chain $\mathcal{M}$ that contain ideals $\{\cal L^{[i]}\}_{i\in\mathbb{N}}$.

Let assume that for some $\mathcal{Q} \in \mathcal{M}$ we have $dim(\mathcal{Q} / \mathcal{Q}_{-})>1.$ The condition $\bigcap\limits_{i=1} ^{\infty}\cal L^{[i]}=\{0\}$ guarantees the existence $s\in \mathbb{N}$ such that $\mathcal{L}^{[s+1]}\subseteq \mathcal{Q}_{-}\subset \mathcal{Q}\subseteq\mathcal{L}^{[s]}.$ Note that pro-solvable of $\cal L$ implies that dimension of $\mathcal{Q} / \mathcal{Q}_{-}$ is finite.

If the kernel of induced family of operators $\widetilde{ad}(\cal L)$ on $\mathcal{Q} / \mathcal{Q}_{-}$ is non-zero, then we can choose $\bar{0}\neq \bar{e}$ from the kernel and the space $\mathcal{Q}_{-}(e)=\langle e \rangle \oplus \mathcal{Q}_{-}$ forms an ideal of $\cal L$, where $\bar{e}=e+\mathcal{Q}_{-}$. It follows that $\mathcal{Q}_{-}\subsetneq \mathcal{Q}_{-}(e) \subsetneq \mathcal{Q}$, which contradicts to maximality of $\cal M$. Let the kernel of $\widetilde{ad}(\cal L)$ on $\mathcal{Q} / \mathcal{Q}_{-}$ is $\{0\}$. Note that $\widetilde{ad}(\cal L)=ad(\widetilde{\cal L}),$ where $\widetilde{\cal L}\subseteq \cal L / \cal L^{[s]}.$ Due to $[ad_x,ad_y]=ad_{[x,y]}$, we deduce that the finite-dimensional Lie algebra of operators $ad(\widetilde{\cal L})$ which acts on finite-dimensional vector space $\mathcal{Q} / \mathcal{Q}_{-}.$ By theorem Lie's theorem (in finite-dimensional case) we can choose common eigenvector $\bar{0}\neq\bar{e}=e+\mathcal{Q}_{-}$, which implies that  $\mathcal{Q}_{-}(e)=\mathcal{Q}_{-}\oplus\langle e \rangle $ forms an ideal of $\cal L$ which strictly lies between $\mathcal{Q}_{-}$ and $\mathcal{Q}_{-}.$
\end{proof}

\begin{prop}\label{bobik} Let $\cal R$ be a pro-solvable Lie algebra.  Then $\cal R^{[2]}$ is pro-nilpotent Lie algebra.
\end{prop}

\begin{proof} For the sake of convenience we denote by $\cal M$ the algebra $\cal R^{[2]}$. The equality $\bigl(\cal R^{[i]}\bigr)^{[j]}=\cal R^{[\,i+j-1\,]}$ for all $i, j \ge1$, implies $\cal M^{[j]}=\cal R^{[\,j+1\,]}.$ From Proposition \ref{prop1.10} we deduce that there exists a maximal chain of ideals $\{\cal V_i\}_{i\in \mathbb{N}}$ such that
$$\cal R =\cal V_1 \supset \cdots \supset \cal V_{i_2-1} \supset \cal M=\cal V_{i_2} \supset \dots \supset \dots \supset \{0\}.$$

Setting $\cal W_i=\cal V_{i_2+i-1}$,  we get $$\cal M=\cal W_{1} \supset \dots \supset \cal W_i \supset \cal W_{i+1} \supset \dots \supset \{0\}.$$
Since $dim(\cal W_i/ \cal W_{i+1})=1$, we have $\cal W_i = \langle e_i\rangle \oplus \cal W_{i+1}$ for some $e_i\in \cal W_i \setminus \cal W_{i+1}$.

Moreover, the embeddings
$ad(\cal R)(\cal W_i) \subseteq \cal W_i$ implies that for $x\in \cal R$ we have

\begin{equation}\label{eq1.1}
ad_x(\cal W_i)=[x,e_i]+\cal W_{i+1}=\lambda^i_x e_i+\cal W_{i+1}
\end{equation}

Let $x,y\in \cal R$, then by the Jacobi identity and \eqref{eq1.1} we derive
$$ad_{[x,y]}(\lambda e_i + \cal W_{i+1})=(ad_x ad_y - ad_yad_x)(\lambda e_i + \cal W_{i+1})=(\lambda^i_x \lambda^i_y - \lambda^i_y \lambda^i_x)e_i + \cal W_{i+1}=\cal W_{i+1}.
$$

Therefore, $\cal M$ is strictly triangularizable Lie algebra. Since $\cal R$ is pro-solvable, so is $\cal M$. By Proposition \ref{prop1.7}, it now suffices to prove that $dim (\cal M^i /\cal M^{i+1}) < \infty $ for any $i\geq 1.$ In fact, the embedding $\cal M^{[i]} \subseteq \cal M^{i}\subsetneq \cal M$ and $dim (\cal M / \cal M^{[i]}) < \infty$ imply that $dim (\cal M / \cal M^{i}) < \infty.$ Finally, by \eqref{eq1}, we conclude that $\cal M$ is pro-nilpotent.  \end{proof}

\begin{cor} Let $\cal R$ triangularizable Lie algebra.  Then $\cal R^{[2]}$ is strictly triangularizable Lie algebra.
\end{cor}

\subsection{Analogue of Engel's theorem}
Let $\cal N$ be a pro-nilpotent Lie algebra. Consider a natural filtration $\{\cal N^i\}_{i\in \mathbb{N}}$ and the finite basis $\cal B_i$ that is preimage of basis the quotient spaces $\cal N^i/ \cal N^{i+1}$. Then we claim that
\begin{equation}\label{eq3.1}
\cal N= \bigoplus\limits_{i=1}^{\infty}Span\{\cal B_i\} \quad \mbox{with the basis} \quad \cal B=\bigcup\limits_{i=1}^{\infty}\cal B_i.
\end{equation}
The linearly independency of elements of $\cal B$ follows from their construction. For the sake of convenience, we set $\cal M=\bigoplus\limits_{i=1}^{\infty}Span\{\cal B_i\}.$
It is clear that  $\cal M \subseteq\mathcal{N}$.  Taking an arbitrary $x\in\mathcal{N}$, one can write
$x = m + x'$ for some $m\in \cal M$ and $x'\notin \cal M.$ The later leads to $x'\in \bigcap\limits_{k=1}^\infty \mathcal{N}^k=0$. Consequently, $\cal M=\cal N$, which implies that $\cal B$ is a basis of $\cal N$ compatible with the lower central series filtration.

In addition, from $\cal N^{i+1}=[\cal N^i, \cal N]$ and $\dim(\cal N/\mathcal{N}^i)<\infty$, we derive the equality
\begin{equation}\label{eq2.2}
[\cal B_i, \cal B_1] =Span\{\cal B_{i+1}\} +\cal N^{i+2}
\end{equation}
for any natural $i$. Therefore, one can take $card(\cal B_{i+1})$ linearly independent elements of the form $[a,b], a\in \cal B_i, b\in \cal B_1$ that lie in $\cal N^{i+1} \setminus \cal N^{i+2}.$ Replacing the elements of $\cal B_{i+1}$ to these elements one can assume that any element of $\cal B_i, i\geq 2$ is a right-normed word of alphabet $\cal B_1$. We call this basis as a {\it natural basis of $\cal N$.}

Let $d$ be a derivation of $\cal N$. Then the matrix form of $d$ on the basis of $\cal N$, which agreed with the natural filtration is given as follows:
\begin{equation}\label{eq222}
D=\begin{pmatrix}
D_{11} & D_{12} & D_{13} & \cdots \\
0 & D_{22} & D_{23} & \cdots \\
0 & 0 & D_{33} & \cdots \\
\vdots & \vdots & \vdots & \ddots
\end{pmatrix}
\end{equation}
here the block $D_{i,j}, i\leq j$ is the matrix representation of the operator $d_{i,j},$ which is restriction and projection of $d$ on subspaces $Span\{\cal B_i\}$ and $Span\{\cal B_j\}$, respectively,

\begin{prop}\label{prop4.1} The derivation $d$ is strictly triangularizable if and only if $d_{11}$ is strictly triangularizable.
\end{prop}
\begin{proof} Let $d$ is strictly triangularizable. Then there exists a maximal chain $\{\cal V_i\}_{i\in \mathbb{N}}$ for $d,$ which refined the natural filtration, then for terms of the chain that contain $\cal N^2$ one can write $\cal V_j=\cal W_j \oplus \cal N^2$ such that $d(\cal V_j)\subseteq \cal V_{j+1},$ which induces a maximal chain $\cal W_j$ for $d_{11}$ such that
$$d_{11}(\cal W_j)\subseteq \cal W_{j+1} \quad \mbox{modulo} \quad \cal N^2,$$
which means that $d_{11}$ is strictly triangularizable.

Suppose that $d_{11}$ is strictly triangularizable. Using a natural filtration $\{\cal N^i\}_{i\in \mathbb{N}}$, we have
$\cal N= \bigoplus\limits_{i=1}^{\infty}Span\{\cal B_i\}$ with the basis $\cal B=\bigcup\limits_{i=1}^{\infty}\cal B_i$. We prove that there exist  basis vectors $\{e_1^{(i)}, \dots , e_{n_i}^{(i)}\}$ on $Span\{\cal B_i\}$ such that $ad_{\cal N}$ is strictly triangularizable. First, we choose basis elements $\{e_1^{(1)}, \dots , e_{n_1}^{(1)}\}$ such that $d_{11}$ is strictly triangularizable. Due to \eqref{eq2.2} we consider define a linear order on the products of already obtained basis elements $\{e_1^{(i)}, \dots, e_{n_i}^{(i)}\}$ to $\{e_1^{(1)}, \dots, e_{n_1}^{(1)}\}$ in the following way:

$$e_l^{(i)}e_k^{(1)}  \preceq e_{l'}^{(i)}e_{k'}^{(1)}$$
if either $k+l<k'+l'$ or $k+l=k'+l'$, $k \leq k'$ with increasing index.

In this linear ordered set we choose increasing linear independent elements denoted by $\{e_1^{(i+1)}, \dots, e_{n_{i+1}}^{(i+1)}\}$ that forms a basis of $Span\{\cal B_{i+1}\}.$

Assume that for a given $i\geq 1$, $d$ is strictly triangularizable in the basis vectors
$\{e_1^{(i)}, \dots , e_{n_i}^{(i)}\}$ on $Span\{\cal B_i\}$.
Setting $U_s^{(i)}=\langle e_s^{(i)}, \dots, e_{n_i}^{(i)}\rangle$, we have
$$d(e_t^{(i+1)})=d([e_l^{(i)},e_k^{(1)}])=
[d(e_l^{(i)}),e_k^{(1)}] + [e_l^{(i)},d(e_k^{(1)})] \subseteq
[U_{l+1}^{(i)}, e_k^{(1)}] + [e_l^{(i)}, U_{k+1}^{(1)}] \subseteq U_{t+1}^{(i+1)}.$$

Putting
$\cal V_{n_1+\dots n_{k-2}+s}=\cal N^{k}\oplus U_s^{k-1}$ we obtain the maximal chain in which $d$ is strictly triangularizable.
\end{proof}

\begin{rem}
It should be noted that in finite dimensional case (see the proof of Theorem 1 in \cite{Snobl}) the statement that $D_{11}$ nilpotent if and only if $D$ is nilpotent the crucial role played the following formula:
\begin{equation}\label{eqder}
d^{r}([x, y])=\sum_{i=0}^{r}\begin{pmatrix} r\\i\end{pmatrix}[d^{i}(x),d^{r-i}(y)], \quad x,y\in \mathcal{L}.
\end{equation}

However, the proof of the analogue statement in terms of strictly triangularizability can not be obtained by using  Equality \eqref{eqder} because of the nilpotency of $D_{11}$ implies the nilpotency of each $D_{ii}$ with increasing nilindices. That is why in Proposition \ref{prop4.1} we have used different approach, which is also hold for finite dimensional case.
\end{rem}

Recall, a linear transformation $T$ of a  vector space $\cal L$ is called {\it diagonalizable}, if $\cal V$ admits a basis $\cal B$ such that every element of $\cal B$ is an eigenvector for $T$.

Taking into account the existence of natural basis as a consequence of Proposition \ref{prop4.1} we have the following result.

\begin{cor} A derivation of a pro-nilpotent Lie algebra is diagonalizable if and only if $d_{11}$ is diagonalizable.
\end{cor}

In the theory of finite-dimensional Lie algebras it is well-known Engel's theorem, which states that a Lie algebra $\cal L$ is nilpotent if and only if $ad_x$ is nilpotent  for any $x\in \cal L$. However, this theorem is not true for infinite dimensional Lie algebras, in general.

\begin{exam} Let $\cal L$  be a space of all strictly upper triangular infinite matrices $M=(m_{i,j})_{i,j\in \mathbb{N}}$ with rows and columns indexed by the natural numbers and with only finitely many non-zero entries $m_{i,j}$ in some ring (due to the finiteness condition, the commutator is well-defined). This commutator Lie algebra is infinite dimensional and it is not nilpotent, but $ad_x$ is nilpotent for each element $x\in M$.
\end{exam}

Below we present an analogue of Engel's theorem for pro-nilpotent Lie algebras.
\begin{thm} \label{Engel's theorem}

$\cal L$ is a pro-nilpotent Lie algebra if and only if $0 < dim (\cal L^i / \cal L^{i+1}) < \infty, \ i\in \mathbb{N}$ and for any $x\in \cal L$ the operators $ad_x$ is strictly triangularizable.
\end{thm}
\begin{proof}  {\bf Sufficiency:} the condition $dim (\cal L^i / \cal L^{i+1}) > 0$ for any $i\in \mathbb{N}$ guarantee the existence of basis $\cal B = \bigcup\limits_{i=1}^{\infty} \cal B_i$ that agreed with the natural filtration. Since for any $x\in \cal L$ operator $ad_x$ is strictly triangularizable, due to Proposition \ref{prop4.1}, it is equivalent that $d_{11}^x$ is strictly triangularizable.

Note that for $x\in \cal L^2$ we have $d_{11}^x=0.$ Therefore, it  suffices to consider $d_{11}^x$ for $x\in Span\{\cal B_1\}.$ Set $\cal B_1=\{e_1, \dots, e_k\}$. Then from $[ad_{e_i}, ad_{e_j}]=ad_{[e_i, e_j]}$ we conclude that $d_{11}^{[e_i, e_j]}=[d_{11}^{e_i}, d_{11}^{e_j}]=0.$ Thus, we get a finite set of pairwise commuting strictly triangularizable operators. From linear algebra, there exists a basis for $Span\{\cal B_1\}$ in which all these operators are simultaneously strictly triangularizable.

 We now utilize the basis vectors introduced in Proposition \ref{prop4.1}, arranged according to the specified linear order. With respect to this basis, it follows that all inner derivations of $\cal L$ are strictly triangularizable. Now, the sufficiency follows from Proposition \ref{prop1.7}.

{\bf The necessity} follows from Proposition \ref{Engelpart1}.
\end{proof}

\section{Maximal torus on pro-nilpotent Lie algebras}

In this section, we adapt classical results on maximal tori to the pro-nilpotent algebras setting  and highlight the new difficulties specific to infinite dimensions.
\begin{defn} A torus on a Lie algebra $\mathcal{N}$ is a commutative subalgebra of $Der(\mathcal{N}),$ consisting of diagonalizable endomorphisms. A torus is said to be maximal if it is not strictly contained in any other torus. We denote by $\mathcal{T}$ a maximal torus on a Lie algebra $\mathcal{N}$.
\end{defn}

Let $\cal N$ be a pro-nilpotent Lie algebra.
Then applying for the characteristic  ideal $\cal N^2$ of $\cal N$ the arguments as in Proposition 3 in \cite{Leger} we derive that dimension of a maximal torus is finite. More, exactly it is bounded from the above by $dim\big(\cal N / \cal N^2\big)$, that is, $rank (\cal N) \leq \dim (\cal N / \cal N^2)$. To avoid confusion in infinite dimensions, we shall expand upon the arguments.

For a given element $t$ of a maximal torus $\cal T$ on $\cal N$ we consider induced derivation
$\bar{t} : \ \cal N/ \cal N^2 \rightarrow \ \cal N/ \cal N^2.$ This defines homomorphism $\varphi$ from maximal torus $\cal T$ on $\cal N$ to a maximal torus $\bar{\cal T}$ on $\cal N/ \cal N^2.$

Let us show that $ker(\varphi)=0.$ Assume that $\varphi(t)=0$ for some $t\in \cal T.$ Since $t$ is diagonalizable, then there exists basis consisting of eigenvectors. Setting $\mathcal{N}_{\alpha}(t)=\{x \in \cal N \mid t(x)=\alpha x\},$ we have decomposition
$$\cal N=\bigoplus\limits_{\alpha\in Spec(t)}\left(\mathcal{U}_{\alpha}(t) \oplus \left(\mathcal{N}_{\alpha}(t) \cap \cal N^2\right)\right).$$

The condition $t\in Ker(\varphi)$ implies that $t(\cal N)\subseteq \cal N^2$, which means $t(\cal U_{\alpha})=0$ for $\alpha\in spec(t)$.  Taking into account that
$Span\{\cal B_1\}=\bigoplus\limits_{\alpha\in Spec(t)}\mathcal{U}_{\alpha}(t),$ we derive $t=0.$
In turn out that $\dim \cal T$ is bounded by dimension of a maximal torus on $\cal N/ \cal N^2$.

Assume $\cal T=Span\{t_1, \dots, t_s\}$. Since any finite set of commuting diagonalizable operators in $Der(\cal N)$ is simultaneously diagonalizable (see Corollary 4.15 in \cite{InfinDiag}). Therefore, we can assume that the pro-nilpotent Lie algebra $\cal N$ admits a basis such that any element of a torus on $\cal N$ has an infinite-dimensional diagonal matrix form.

Let $\cal N$ be a pro-nilpotent Lie algebra with countable Hamel basis $\cal B=\{e_i \ | \ i\in \mathbb{N}\}$ and the multiplications table of $\cal N$ is determined by laws on products of basic elements:
\begin{equation}\label{eqtable}
[e_i,e_j]=\sum\limits_{p\in \cal I_{i,j}}^{}\gamma_{i,j}^pe_p, \quad i,j\in \mathbb{N},
\end{equation}
where subsets $\cal I_{i,j}$ of $\mathbb{N}$ are finite for each $i,j\in \cal N$.

For $i, j, p$ such that $\gamma_{i,j}^p\neq0$ we consider the system of the infinitely many linear equations
$$S_{\cal B}: \quad \left\{\alpha_{i}+\alpha_{j}=\alpha_{p},
\right.$$
in the variables $\{\alpha_i \ | \  i\in \mathbb{N}\}.$

Let $d=diag(\alpha_1, \alpha_2, \dots )$ be a diagonal transformation of $\cal N$. Then due to equalities
$$\sum\limits_{p\in \cal I_{i,j}}^{}\gamma_{i,j}^p\alpha_pe_p=d([e_i,e_i])=[d(e_i),e_j] + [e_i,d(e_j)]=(\alpha_i+\alpha_j)[e_i,e_j]=
(\alpha_i+\alpha_j)\sum\limits_{p\in \cal I_{i,j}}^{}\gamma_{i,j}^pe_p,$$
we obtain that $d$ is a derivation of $\cal N$ if and only if $\alpha_i$  are solutions of the system $S_{\cal B}$.

Following the paper \cite{Leger} we denote by $r\{\cal B\}$ the rank of the system $S_{\cal B}$ (the dimension of the fundamental system of solutions). Pro-nilpotency of $\cal N$ guarantee that the rank of $S_{\cal B}$ is finite. Setting
$r\{\cal N\}=\max r\{\cal B\}$ as $\{\cal B\}$ runs over all bases of $\cal N$. Similarly to the finite-dimensional case (see \cite{Leger}) for a pro-nilpotent Lie algebra $\cal N$ over an algebraically closed field the equality $dim \mathcal{T}=r\{\cal N\}$ holds true.

We denote the free parameters in the solutions of the system $S_{\cal B}$ by $\alpha_{i_1},\dots, \alpha_{i_s}$. Making renumeration the basis elements of $\mathcal{N}$ we can assume that $\alpha_{1},\dots, \alpha_{s}$ are free parameters of $S_{\cal B}$. Then we get
\begin{equation}\label{eq3.6}
\alpha_{i}=\sum\limits_{t=1}^{s}\lambda_{i,j}\alpha_{t},\ i\geq s+1.
\end{equation}

Set the basis $\{\alpha_i=(\alpha_{1,i}, \alpha_{2,i},  \dots) \ |  \ 1\leq i\leq s\}$ of fundamental solutions of the system $S_{\cal B}$. Then the diagonal matrices $\diag(\alpha_i)$ form a basis of a maximal torus of $\cal N$.

Now, let $\cal N$ be  a pro-nilpotent Lie algebra with a  maximal torus $\cal T$. Suppose $\cal T$ is simultaneously diagonalizable in the basis $\cal B=\bigcup\limits_{i=1}^{\infty}\cal B_i$. Define $\cal W_k=\cal N / \cal N^{k+1}.$ Then $\cal W_k$ is a finite-dimensional nilpotent Lie algebra for any $k$. We denote by $\cal T_{k}$ a maximal torus of $\cal W_k$ in the basis $\bigcup\limits_{i=1}^{k} \cal B_i$.

The family of tori $\{\cal T_k\}_{k\geq k_0}$ is compatible in the following sense: for every $k\geq k_0$, the map $\cal T_{k+1}$ agrees with $\cal T_k$ when restricted via the natural projection $\cal W_k \to \cal W_{k+1}$. Then, under the consistency condition above, there exists a unique extension $\cal T$
such that for every $k$, the restriction of $\cal T$ to $\cal W_k$ coincides with $\cal T_k$. It is naturally regarded as the {\it limit} of the consistent system of transformations $\{\cal T_k\}$ on the inverse system of factor spaces $\cal W_k$ and we get  $\lim\limits_{k \rightarrow \infty} \cal T_{k}=\cal T$.

Consider a pro-nilpotent Lie algebra $\cal N$ and its two maximal tori $\cal T, \widetilde{\cal T}$. Note that $\widetilde{\cal T}$ is simultaneously diagonalizable in the basis $\widetilde{\cal B}=\bigcup\limits_{i=1}^{\infty} \widetilde{\cal B}_i$ and $\widetilde{\cal T}_{k}$ is a maximal torus of $\cal W_k$ in the basis $\bigcup\limits_{i=1}^{k} \widetilde{\cal B}_i.$

Then we have $\lim\limits_{k \rightarrow \infty} \cal T_{k}=\cal T$ and $\lim\limits_{k \rightarrow \infty} \widetilde{\cal T}_{k}=\widetilde{\cal T}.$ Taking into account that dimensions of tori $\cal T$ and $\widetilde{\cal T}$ are finite, we obtain the existence
$q_1$ and $q_2$ such that $\dim \cal T_{q_1} = \dim \cal T$ and $\dim \widetilde{\cal T}_{q_2} = \dim \widetilde{\cal T}.$ Taking $q=max\{q_1, q_2\}$ we get
$$\dim \cal T_{s} = \dim \cal T, \quad \dim \widetilde{\cal T}_{s} = \dim \widetilde{\cal T} \quad \mbox{for any} \quad s\geq q.$$

Now, applying the corollary of Mostow's result (see Theorem 4.1 in \cite{Mostow}) for a finite-dimensional nilpotent Lie algebra $\cal N$ which states that any two maximal tori on $\cal N$ are conjugate under an inner automorphism of $\cal N$, we deduce
$\dim \cal T_q =\dim \widetilde{\cal T}_q$. Consequently, $\dim \cal T=\dim \widetilde{\cal T}.$

Thus, we proved the following result.

\begin{thm} Let $\cal N$ be a pro-nilpotent Lie algebra and $\cal T$ be its a maximal torus. Then $\lim\limits_{k \rightarrow \infty} \cal T_{k}=\cal T,$ where $\cal T_{k}$ are maximal tori of algebras $\cal N/ \cal N^k.$ Moreover, dimensions of arbitrary two maximal tori are equal.
\end{thm}

Similarly to the finite dimensional case, we call {\it rank of the algebra $\cal N$} (denoted further by $rank(\cal N)$) the dimension of its a maximal torus.

For a pro-nilpotent Lie algebra and a maximal torus, one can construct a residually solvable algebra.

\begin{exam} \label{exam3.11}

Let $\cal N$ be a pro-nilpotent Lie algebra and $\cal T$ a maximal torus on $\cal N$. Define the Lie algebra $\cal {R}_{\cal T} = \cal N \rtimes \cal T$, with the Lie bracket defined as follows:
$$
[x, t] = t(x), \quad [t, t'] = 0 \quad \text{for all } x \in \cal N \text{ and } t, t' \in \cal T.
$$
It follows that $\cal R_{\cal T}^2 \subseteq \cal N$. Since $\cal T$ is finite-dimensional, the quotient ${\cal R}_{\cal T} / {\cal R}_{\cal T}^2$ is also finite-dimensional.

Furthermore, we have the inclusion:
$$
\cal {R}_{\cal T}^{[i+1]} \subseteq \cal N^{[i]} \subseteq \cal N^{2^{i-1}},
$$
which implies that ${\cal R}_{\cal T}$ is a residually solvable Lie algebra whose pro-nilpotent radical is $\cal N$.
\end{exam}

Note that in general $[\cal N, \cal T] \subsetneq \cal N.$ For instance, the direct sum of a finite-dimensional nilpotent Lie algebra with zero root space and an arbitrary pro-nilpotent Lie algebra poses this property. Another example, which does not have structure of the direct sum of two algebras, of pro-nilpotent Lie algebra such that $[\cal N, \cal T] \subsetneq \cal N$ is given below.

\begin{exam} Consider the Lie algebra $\cal N=\mathfrak{m}_2\rtimes \cal N_8$, where $\mathfrak{m}_2=Span\{f_i \ | \ i\in \mathbb{N}\}$ and  $\cal N_8=Span\{e_1, \dots, e_8\}$, which is given by its table of multiplications
$$\left\{\begin{array}{llllllll}
[e_{1}, e_{2}] = e_{3},&[e_{1}, e_{3}] = e_{4},&[e_{1}, e_{4}] = e_{5},&[e_{1}, e_{6}] = e_{7},&[e_{1}, e_{8}] = e_{9},\\[1mm]
[e_{2}, e_{3}] = e_{8},&[e_{2}, e_{4}] = e_{9},&[e_{2}, e_{5}] = e_{9},&[e_{4}, e_{3}] = e_{9},\\[1mm]
[e_1,f_i]=f_{i+1},& i\geq 1,\\[1mm]
[f_1,f_i]=f_{i+2},& i\geq 2.\\[1mm]
\end{array}\right.$$
Then $\cal N$ is pro-nilpotent and
$\cal T=Span\{diag(0,\alpha,\alpha,\alpha,\alpha,\beta,\beta,
2\alpha,2\alpha,0,0,0,\dots)\ | \ \alpha, \beta \in \mathbb{C}\}$ is its a maximal torus. Therefore, we have $[\cal N, \cal T] \subsetneq \cal N$ and the algebra $\cal R_{\cal T}$ has the following multiplications table:
$$[\cal N,\cal N], \quad [e_i,x]=e_i, \ 2\le i\le 5, \quad [e_i,y]=e_i, \ 6\le i\le7, \quad [e_i,x]=2e_i, \ 8\le i\le 9.$$

By virtue of $\cal R_{\cal T}^{[3]}=\{0\}$, we conclude that
$\cal R_{\cal T}$ non pro-solvable.
\end{exam}


The works of \cite{Residual1} and \cite{Residual2}  provide examples of pro-nilpotent Lie algebras for which $[\cal N, \cal T]=\cal N$.

Moreover, if $\cal N$ is a pro-nilpotent Lie algebra that satisfies the conditions: $[\cal N, \cal T]=\cal N$ and $dim \big(\cal N/ \cal N^{[i]}\big) < \infty, \ i\geq 2$, then $\cal R_{\cal T}$ is pro-solvable Lie algebra. Indeed, from
$$\cal R_{\cal T} / {\cal R_{\cal T}^{[i+1]}}=
(\cal N\oplus \cal T) / \cal R_{\cal T}^{[i+1]}=
\cal N/ \cal R_{\cal T}^{[i+1]}\oplus \cal T=
\cal N / \cal N^{[i]}\oplus \cal T$$
and $dim \cal T<\infty$, we conclude that
$dim \big(\cal R_{\cal T}/{\cal R_{\cal T}^{[i]}}\big)<\infty$ for any $i\geq 2.$

Let $\cal T=\operatorname{Span}\{t_{1}, t_{2}, \dots, t_{s}\} $ be a maximal torus on a pro-nilpotent Lie algebra $\cal N$ and $\cal B$ a Hamel basis of $\cal N$ such that all $t_i, \ 1\leq i \leq s$ have diagonal matrix form with respect to this basis. Let $\alpha$ be an element of the dual space $\cal T^*$. Then $\alpha=\sum\limits_{i=1}^s\alpha_i t_i^*,$ where $t_i^*$ is the standard basis of $\cal T^*$ associated with the basis $t_i$ (that is $t_i^*(t_j)=\delta_{i,j}$). We set
$$ \cal N_{\alpha}=\{x\in \cal N \ | \ t(x)=\alpha(t)x \ \mbox{for any} \ t\in \cal T\}.$$

Obviously,
$\cal N_{\alpha}\cap \cal N_{\beta}=0.$ Indeed, if $x\in \cal N_{\alpha}\cap \cal N_{\beta}$ for $\alpha\neq \beta,$ then $t_i(x)=\alpha(t_i)x=\beta(t_i)x$ for any $1\leq i \leq s.$ Consequently, $x=0$.

Denote by $W=\left\{\alpha\in {\cal T}^\ast \ : \ \cal N_\alpha\neq 0\right\}$ the roots system of  $\cal N$  associated to  $\cal T.$ We have $\cal N \supseteq \mathcal{H}\{\bigoplus\limits_{\alpha\in \cal W}\cal N_{\alpha}\},$
where $x\in \mathcal{H}\{\bigoplus\limits_{\alpha\in \cal W}\cal N_{\alpha}\}$ means that $x$ is written as a finite sum of elements from $\mathcal{N}_{\alpha}, \alpha\in {\cal T}^\ast$.

Let $x=\sum\limits_{i=p}^{q}x_ie_i$ be an element of $\cal N.$ Then from
$t_i(e_j)=\alpha_j(t_i)e_j$ we have
$t_i(x)=\sum\limits_{j=p}^{q}\alpha_j(t_i)x_je_j.$ For those equal functionals $\alpha_j$, we conclude that $x_je_j$ lie in the same root space $\cal N_{\alpha_j}$. Hence,
$x\in \bigoplus\limits_{j=p'}^{q'}\cal N_{\alpha_j},$ which implies $\cal N=\mathcal{H}\{\bigoplus\limits_{\alpha\in \cal W}\cal N_{\alpha}\}.$

Note that $\cal T$ admits a basis $\{t_1, \dots, t_s\}$ which satisfies the condition
\begin{equation}\label{eq3.1.1}
t_i(x_{\alpha_j})=\alpha_{i,j}x_{\alpha_j}, 1\leq i,j\leq s,
\end{equation}
where $\alpha_{i,i}=1, \ \alpha_{i,j}=0$ for $i\neq j$ and $x_{\alpha_j}\in \cal N_{\alpha_j}.$

Denote by $\Psi=\{\alpha_1, \dots, \alpha_s, \alpha_{s+1}, \dots, \alpha_k\}$ the set of simple roots of $W$ such that any root subspace with index in $\Psi$ contains at least one generator element of $\cal N$ and by $\Psi_1=\left\{\alpha_1, \ldots, \alpha_s\right\}$ the set of primitive roots such that any non-primitive root can be expressed by a linear combination of them. In fact, for any root $\alpha\in W$ we have $\alpha=\sum\limits_{\alpha_i\in \Psi_1}p_i\alpha_i$ with $p_i\in \mathbb{Z}.$

Despite to finite-dimensional case, for the case of pro-nilpotent Lie algebra dimensions of root spaces not need to be finite. For instance,
for the direct sum of the characteristically pro-nilpotent Lie algebra from Example \ref{exam3.8} and the $n$-dimensional abelian Lie algebra with respect to a maximal torus we have root space decomposition
$$\cal N_{0}\oplus\cal N_{\alpha_1}\oplus \cdots \oplus \cal N_{\alpha_n} \quad \mbox{with} \quad \dim\cal N_{\alpha_i}=1 \quad \mbox{and} \quad \dim\cal N_{0}=\infty.$$

The result on conjugacy of two maximal tori on a finite-dimensional nilpotent Lie algebra $\cal N$ states that any two maximal tori are conjugated via an automorphism of the form $exp(d)$, where $d$ is a nilpotent derivation of $\cal N.$ However, in infinite dimensional case, we do not even  have the convergence of $exp(d)=\sum\limits_{n=0}^\infty \tfrac{1}{n!}d^n$ for an arbitrary derivation. The next result shows that in the case of pro-nilpotent complete Lie algebras this series is uniformly converges and it is an automorphism.

\begin{thm} Let $\cal N$ be a pro-nilpotent complete Lie algebra and $d$ be a derivation of $\cal N$ then the series $\sum\limits_{n=0}^\infty \tfrac{d^n(x)}{n!}$ uniformly convergent on $\cal N$. As usual, the limit function is denoted by $\exp{(d)}$. Moreover, $\exp(d)$ is an  automorphism of $\cal N.$
\end{thm}
\begin{proof} For an arbitrary $x\in\mathcal{N}$ and a derivation of $\cal N$ we set
$S_N(x)=\sum\limits_{n=0}^N \tfrac{d^n(x)}{n!}.$
Since $d^n(x)\in\mathcal{N}^n$, the remainder
$S_M(x)-S_N(x)=\sum\limits_{j=N+1}^M\tfrac{d^j(x)}{j!}$
lies in $\mathcal{N}^{N+1}$ for $M>N$.  Thus, the sequence $\{S_N(x)\}_{N}$ is Cauchy and converges by completeness and the limit function denoted $\exp{d}$. Since $\dim (\cal N/ \cal N^n)<\infty$, one gets $N$ does not depend on choosing $x$.  This gives us $\exp{(d)}$ is also continuous. Each finite partial sum \(S_N\) is a linear operator, and the limit of linear maps in this topology is again linear.  Hence $\exp(d)\colon\cal N\to\cal N$ is linear. Now we check the invertibility of $\exp(d)$.

Take two formal power series operators on $\cal N$:
   $$\exp(d) \;=\;\sum_{m=0}^\infty \frac{d^m}{m!},
     \qquad \exp(-d) \;=\;\sum_{n=0}^\infty \frac{(-d)^n}{n!}$$
and compose them, one gets
$$\exp(d)\,\circ\,\exp(-d)=\Bigl(\sum_{m=0}^\infty \tfrac{d^m}{m!}\Bigr)\circ\Bigl(\sum_{n=0}^\infty \tfrac{(-d)^n}{n!}\Bigr).$$
Formally, this double series expands to
$$\sum_{m=0}^\infty\sum_{n=0}^\infty\frac{d^m}{m!}\,\frac{(-d)^n}{n!}=\sum_{k=0}^\infty \Bigl(\sum_{m+n=k}\frac{(-1)^n}{m!\,n!}\Bigr)\,d^k=
id + \sum_{k=1}^\infty \Bigl(\frac{(1-1)^k}{k!}\Bigr)\,d^k=id.$$
Finally, because $d$ satisfies the Leibniz rule
$d([x,y])=[d(x),y]+[x,d(y)]$, one shows by the usual power series ``conjugation'' argument that
$$\exp(d)\bigl([x,y]\bigr)=\bigl[\exp(d)(x),\,\exp(d)(y)\bigr] \quad \mbox{for any} \quad x,y\in \cal N.$$
Thus, \(\exp(d)\) is an automorphism of $\cal N$.
\end{proof}

\begin{rem}
It is well-known that in finite-dimensional nilpotent Lie algebras, any two maximal tori are conjugate under an inner automorphism. While the existence of such a conjugating automorphism remains an open question for general pro-nilpotent algebras, the exponential map can define automorphisms for complete pro-nilpotent Lie algebras. For non-complete pro-nilpotent Lie algebra in the natural filtration topology, the convergence of the series $\operatorname{exp}(d)$ depends on derivations.
\end{rem}

\section{On residually solvable extensions.}

This section devoted to the study of residually solvable extensions of pro-nilpotent Lie algebras, generalizing the finite-dimensional solvable extensions. We define the pro-nilpotent radical and establish conditions for its existence. Furthermore, we provide a method, complete with a precise algorithm, to construct residually solvable extensions under certain conditions. Special attention is devoted to maximal extensions of pro-nilpotent algebras of maximal rank and to characteristically pro-nilpotent algebras, which analogously to the finite-dimensional case, do not admit nontrivial solvable extensions.

Due to Example \ref{exam2.6}, the maximal (by inclusion) pro-nilpotent ideal not always exist. In case of existence the pro-nilpotent ideal which contain all pro-nilpotent ideals of $\cal R,$ then we call this ideal {\it pro-nilpotent radical}.

\begin{prop}\label{prop5.1}  \

(i) There exists the pro-nilpotent radical in a pro-solvable Lie algebra $\cal R$;

(ii) In arbitrary residually solvable Lie algebra $\cal R$  satisfying the conditions: $\dim (\cal R/\cal R^2)< \infty$ and $\cal R^2$ is pro-nilpotent ideal, there exists the pro-nilpotent radical.
\end{prop}
\begin{proof} (i)
By Proposition \ref{bobik} we deduce that $\cal R^2$ is pro-nilpotent ideal of the algebra $\cal R$. Since $\dim (\cal R/\cal R^2)< \infty$, then there are finitely many pro-nilpotent ideals which contain $\cal R^2$. Now applying finitely many times Lemma \ref{lem:pro-nilpotent-sum}, we obtain that the sum of all pro-nilpotent ideals which contain $\cal R^2$ is a pro-nilpotent ideal, which we denote by $\cal M$. Consider an arbitrary pro-nilpotent ideal $\cal I$ of $\cal R$. Then again by Lemma \ref{lem:pro-nilpotent-sum} we get pro-nilpotent ideal $\cal R^2 + \cal I$ which contains $\cal R^2$. Since $\cal I\subseteq \cal R^2 + \cal I\subseteq\cal M$, we conclude that any pro-nilpotent ideal contained in $\cal M$. Therefore, $\cal M$ is the pro-nilpotent radical of the algebra $\cal R$.

Part (ii) is proved by applying the same arguments as in part (i).
\end{proof}

\subsection{Residually solvable extensions of pro-nilpotent Lie algebras.}
\noindent
By analogy with the finite-dimensional case, we will consider residually solvable Lie algebras with a given pro-nilpotent radical. According to Proposition \ref{prop5.1}(ii), this is possible whenever the residually solvable extension meets specific restrictions.

\begin{defn} For a given pro-nilpotent Lie algebra $\cal N$, one defines
$$\cal M_{\cal N}=\{\cal R: \cal N \ \textrm{pro-nilpotent radical of} \ \mathcal{R}, \ \dim \big(\cal R/ \cal R^2\big) < \infty\ \textrm{and}\ \cal R^2 \subseteq \cal N\}.$$
Since $\cal M_{\cal N}$ consists of residually solvable Lie algebras whose pro-nilpotent nilradical is $\cal N$ with $codim \cal N < \infty$, we refer an arbitrary element of $\cal M_{\cal N}$ as a residually solvable extension of $\mathcal{N}$. Furthermore, a residually solvable extension $\cal R$ of $\cal N$ is said to be a maximal if $\dim(\cal R / \cal N)$ is maximal.
\end{defn}

The following result provides a base of a method for construction under certain conditions of residually solvable extensions of pro-nilpotent Lie algebras, which is an extension of finite-dimensional method for the description of solvable algebras with a given nilradical \cite{Mubor}, \cite{Snobl}.

\begin{thm} \label{prop1.20} For an arbitrary $\mathcal R\in \cal M_{\cal N}$ the family $ad(\cal R)_{|\cal N}$ is triangularizable, while ${ad_x}_{|\cal N}$  is non strictly triangularizable for any $x\in \cal R \setminus \cal N$.
\end{thm}

\begin{proof} Consider the chain of subspaces of $\cal N$ which are invariant under $ad(\cal R)$ (because of being ideals)
\begin{equation}\label{chain0}
\mathcal{N} \supset \mathcal{N}^2 \supset \cdots \supset \mathcal{N}^i \supset \mathcal{N}^{i+1} \supset \cdots \supset \{0\},
\end{equation}

Following the approach used in the proof of Proposition \ref{prop1.10}, we consider a maximal chain $\mathcal{M}$ of ideals of $\mathcal{R}$ that lie in $\cal N$ and contains the chain in~\eqref{chain0}. Suppose, for the sake of contradiction, that there exists an ideal $\mathcal{W} \in \mathcal{M}$ such that $\dim(\mathcal{W} / \mathcal{W}_{-}) > 1$, where $\mathcal{W}_{-}$ denotes the immediate predecessor of $\mathcal{W}$ in the chain.

Since $\cal N$ is pro-nilpotent, there exists an integer $s \in \mathbb{N}$ such that

$$\cal N^{s+1} \subseteq \mathcal{W}_{-} \subset \mathcal{W} \subseteq \cal N^{s} \quad \mbox{and} \quad \dim(\mathcal{W} / \mathcal{W}_{-}) < \infty.$$

Therefore, we have induced family of operators $\widetilde{ad}(\cal R) \ : \ \cal W / \cal W_{-} \rightarrow \cal W / \cal W_{-}.$

The conditions of the theorem imply the decomposition $\mathcal{R} = \mathcal{N} \oplus \mathcal{Q}$, where $\mathcal{Q}$ is a finite-dimensional complementary subspace to $\mathcal{N}$. Taking into account that $\widetilde{ad}(\cal N^{s+1})=0$ we obtain that $\widetilde{ad}(\cal R) \subseteq \widetilde{ad}(\cal Q \oplus Span\{e_1, \dots, e_k\}),$ where $\{e_1, \dots, e_k\}$ are basis elements of $\cal N$ that do not lie in $\cal W_{-}.$ This implies that $\widetilde{ad}(\cal R)$ is finite-dimensional Lie algebra. Moreover, from $\cal R^2 \subseteq \cal N$ we derive that $\cal R^{[i+1]} \subseteq \cal N^{[i]} \subseteq \cal N^{2^{i-1}}.$ By virtue of pro-nilpotency of $\cal N$ we obtain the solvability of the Lie algebra $\widetilde{ad}(\cal R)$. Now, applying Lie's theorem (in finite-dimensional case) we can choose common eigenvector $\bar{0}\neq\bar{e}=e+\mathcal{W}_{-}$, which implies that  $\mathcal{W}_{-}(e)=\mathcal{W}_{-}\oplus\langle e \rangle $ forms an ideal of $\cal R$ which strictly lies between $\mathcal{W}_{-}$ and $\mathcal{W}.$

Now, we are going to prove that ${ad_x}_{| \cal N}$ is non strictly triangularizable for any $x\notin \cal N$. Let us assume the contrary, that is, there exists $x\notin\mathcal N$ such that $\mathrm{ad}_{x|\mathcal N}$ is strictly triangularizable. Set $\mathcal V = \mathcal N \oplus \langle x\rangle$.  Since $\mathcal R^2\subseteq\mathcal N$, it follows that $\mathcal V$ is an ideal of~$\mathcal R$.

Since ${ad_x}_{|\cal N}$ is strictly triangularizable, then for an arbitrary fixed $s\in\mathbb N$ there exists $l(s)\in\mathbb N$ such that $\mathrm{ad}_x^{l(s)}(\mathcal N)\;\subseteq\;\mathcal N^s$.

Now, we show that for each $s$,
$$\mathcal N^{(s+1)\,l(s)}
\;\subseteq\;
(\mathcal N\oplus\langle x\rangle)^{(s+1)\,l(s)}
\;\subseteq\;
\mathcal N^s.$$

The first inclusion is obvious.  To prove the second, expand a typical term in
$(\mathcal N\oplus\langle x\rangle)^{(s+1)\,l(s)}.
$
If $\mathcal N$ appears at least $s$ times, then since each $\mathcal N^i$ is an ideal, the term lies in~$\mathcal N^s$.  Otherwise, $\mathcal N$ appears at most $(s-1)$ times, then we have only two distinguishes cases.

{\bf Case 1.} There are two consecutive factors of~$\mathcal N$ with at least $l(s)$ copies of~$x$ between them. Then by the choice of $l(s)$ we have the embedding $\mathrm{ad}_x^{\,l(s)}(\mathcal N)\subseteq\mathcal N^s,$ which forces the term into~$\mathcal N^s$.

{\bf Case 2.} If no two copies of~$\mathcal N$ are separated by $l(s)$ or more~$x$'s, then by the pigeonhole principle there must be a run of $l(s)$ consecutive~$x$'s at one end of the term, and again
$\mathrm{ad}_x^{\,l(s)}(\mathcal N)\subseteq\mathcal N^s$
  shows the term lies in~$\mathcal N^s$. Thus, we obtain
$$\bigcap_{s=1}^\infty \cal V^{s}=
\bigcap_{s=1}^\infty \cal V^{(s+1)\,l(s)}
\;\subseteq\; \bigcap_{s=1}^\infty \mathcal N^s \;=\; \{0\}, $$
while
\[\dim\frac{\cal V}{\cal V^{(s+1)\,l(s)}} \;\le\;
1 + \dim\frac{\mathcal N}{\mathcal N^{(s+1)\,l(s)}}
\;<\; \infty. \]
The arbitrariness of $s$ and \eqref{eq1} imply that $\cal V$ is the pro-nilpotent ideal strictly contains $\mathcal N$, contradicting the maximality of the pro-nilpotent radical $\cal N$. Therefore, ${ad_x}_{|\cal N}$ is not strictly triangularizable.
\end{proof}

\begin{defn} Derivations $d_1, \dots, d_s$ of $\cal N$ are called strictly triangularizable independent (ST-independent), if the map $d=\alpha_1d_1+\cdots + \alpha_s d_s$ is not strictly triangularizable for any non-zero scalars $\alpha_i, 1\leq i \leq s$. In other words, $d$ is strictly triangularizable if and only if $\alpha_i=0, 1\leq i \leq s.$
\end{defn}

From Theorem \ref{prop1.20} we conclude that the codimension of $\cal N$ in $\cal R$ is not greater than the maximal number of ST-independent derivations of $\cal N$.

{\bf Conclusion:} The results of Theorem \ref{prop1.20}  lead that the method of construction of finite-dimensional solvable Lie algebras with a given nilpotent (non-characteristically nilpotent) algebras can be extended under certain conditions to the case of residually solvable extensions of a given pro-nilpotent Lie algebra. Namely, the algorithm consists of the following steps:
\begin{itemize}
\item[-] to describe the space $\operatorname{Der}(\cal N)$ for a given pro-nilpotent Lie algebra $\cal N;$

\item[-] to allocate all ST-independent derivations, say, $d_1, \dots, d_k;$

\item[-] setting $\cal Q=\operatorname{Span}\{x_1, \dots, x_s\}$ with the action of ${\operatorname{ad}_{x_i}}_{| \cal N}=d_{i_j}$ with $j_i\in\{1, \dots, k\}$ get the product $[\cal N, \cal Q];$

\item[-] applying Jacobi identity and suitable basis
transformations of the space $\cal N\oplus \cal Q$ to describe the product $[\cal Q, \cal Q].$
\end{itemize}

In fact, the descriptions of residually solvable extensions of some pro-nilpotent Lie algebras in  \cite{Residual1}-\cite{Residual2} these steps are used.

Recall that characteristically nilpotent Lie algebra is characterized the nilpotency of all its derivations. For more detailed results on characteristically nilpotent Lie algebras we refer the reader to \cite{CharNil} and references therein. This class of nilpotent Lie algebras does not admit solvable extensions, that is, there is no solvable Lie algebra whose nilradical is characteristically nilpotent algebra. This inspired us to allocate an analogous class of pro-nilpotent Lie algebras whose an arbitrary derivation is strictly triangularizable by meaning that such algebras will play the same role as characteristically nilpotent Lie algebras in finite dimensional case.

\begin{defn} An algebra $\cal N$ which satisfies the condition $0 < \dim\big(\cal N^i/ \cal N^{i+1}\big) < \infty, \ i\in \mathbb{N}$ is called characteristically pro-nilpotent Lie algebra if
any $d\in Der(\cal N)$ is strictly triangularizable.
\end{defn}

Since $ad_x$ for any $x$ element characteristically pro-nilpotent Lie algebra $\cal N$ is strictly triangularizable, from Theorem \ref{Engel's theorem} we conclude that characteristically pro-nilpotent Lie algebra is pro-nilpotent one.

Below we present an example of characteristically pro-nilpotent Lie algebra.
\begin{exam} \label{exam3.8} Consider the following Lie algebra:
$$\cal N:\left\{\begin{array}{llll}
[e_i,e_1]=e_{i+1},&i\geq2,\\[1mm]
[e_i,e_2]=e_{i+2}+e_{i+3},&i\geq3.\\[1mm]
\end{array}\right.$$

Obviously, $\dim\big(\cal N/ \cal N^{2}\big)=2$ and $\dim\big(\cal N^i/ \cal N^{i+1}\big)=1$ for $i\geq 2.$ Moreover, by straightforward computation one can show that any derivation of the algebra $\cal N$ is strictly triangular. Hence, the algebra $\cal N$ is characteristically pro-nilpotent.
\end{exam}

Thus, the above steps of algorithm do not applicable for characteristically pro-nilpotent algebras because does not exist any non-strictly triangularizable derivations of such algebras.

\begin{thm}\label{thm3.17}
For an arbitrary $\mathcal R\in \cal M_{\cal N}$
we have $\dim (\cal R/ \cal N)\leq\cal \dim \big(\cal N/\cal N^{2}\big).$

\end{thm}
\begin{proof} Consider the vector spaces decomposition $\cal R=\cal N\oplus \cal Q$. From the conditions of theorem it follows that $\dim \cal Q < \infty.$ Let $d$ be a derivation of $\mathcal{N}$ and $\cal B=\bigcup\limits_{i=1}^{\infty}\cal B_i$ of the nilpotent algebra $\cal N$ that is agreed with natural filtration.

The relation $d([a,b])=[d(a),b]+[a,d(b)]$ gives that the elements of the diagonal blocks $D_{ll},$ $(l=2,\dots, q)$ of $D$ are linear functions of elements of $D_{11}$ and the structure constants of $\mathcal{N}$.

The condition $\mathcal{R}^2\subseteq\mathcal{N}$ implies $[ad_{z_1},ad_{z_2}]=ad_{[z_1,z_2]}\in Ad(\mathcal{N})$
for any $z_{1},z_{2}\in\mathcal{Q}.$

Consider a set of derivations $\{D_{1},\dots,D_{l}\}$ of $\mathcal{N}$ such that $[D_i,D_j]\in Ad(\mathcal{N})$ for any $i,j\in \{1, \dots, l\}.$ This gives

\begin{equation} \label{eq2.5}
\begin{array}{cccc}
[D_{ll}^{i},D_{ll}^{j}]=0, && \ \mbox{where}\ 1\leq l \leq q, \ 1 \leq t \leq p \ \mbox{and}\ 1 \leq i,j \leq l.
\end{array}
\end{equation}

From Proposition \ref{prop4.1} we get that the derivations $D_1,\dots, D_l$ are ST-independent  if and only if their corresponding submatrices
$\{D_{11}^{1},\dots,D_{11}^{l}\}$ are ST-independent (which is equivalent to linearly nil-independency because of finite size of the first diagonal block).

Along with \eqref{eq2.5} this implies that the number of ST-independent outer derivations of $\cal N$ is bounded from the above by the maximal number of ST-independent pairwise commuting matrices of the size $k\times k$, where $k=Card(\cal B_1)$. This number is equal to $\dim \cal N- \dim \big([\cal N,\cal N]\big)=dim \big(\cal N / \cal N^2\big),$ which is the needed estimation.
\end{proof}

It should be noted that the conditions of Theorem \ref{thm3.17} imply that $\cal R$ is residually solvable Lie algebra. However, these conditions not always true for an arbitrary residually solvable Lie algebra. In the following example we present a residually solvable Lie algebra $\cal R$ that do not satisfies the conditions $\dim \big(\cal R / \cal R^2\big) < \infty$ and the existence of pro-nilpotent radical of $\cal R$.

\begin{exam} Let $\cal R'$ be a Lie algebra that satisfies the condition of Theorem \ref{thm3.17} (for instance, $\cal R'$ is non negative part of Witt algebra with a basis $\{e_i \ | \ i\in \mathbb{N}\cup 0\}$). Consider $\cal R=\cal R'\oplus \cal A,$ where $\cal A$ is countable dimensional abelian algebra. Then $\cal R$ is residually solvable Lie algebra that does not have pro-nilpotent radical and $\dim \big(\cal R/ \cal R^2\big)=\infty.$
\end{exam}

When a maximal residually solvable extension exists, the following result gives the exact codimension of the pro-nilpotent radical.

\begin{prop} Let $\mathcal{R}\in \cal M_{\cal N}$ be a maximal residually solvable extension. Then $rank(\cal N)=\operatorname{codim} \cal N.$
\end{prop}
\begin{proof} Let $\mathcal{R}=\cal N \oplus \cal Q$ and $\{x_1, \dots, x_s\}$ be a basis of $\cal Q.$ If $s < rank (\cal N)$, then $\cal R$ is not a maximal extension of $\cal N$ because of $\codim \cal N=\ dim \cal T$ in $\dim\cal R_{\cal T}$ is strictly greater than $\codim \cal N=\dim \cal Q$ in $\cal R$. Thereby, $s \geq rank (\cal N).$

Thanks to Theorem \ref{prop1.20} one can conclude that
there exist a basis $\cal B$ of $\cal N$ such that $ad(\cal Q)_{|\cal N}$ are triangular but not strictly triangular. For the sake of convenience, we shall assume that  $ad(\cal Q)_{|\cal N}$ in the basis have infinite upper-triangular matrix forms. In order to diagonal of ${ad_{x}}_{|\mathcal{N}}, \ x\in \cal Q$, denoted further by $d_0^x$, is a derivation of $\cal N$ it has to be solution of the system of infinitely many linear equations $S_{\cal B}$. Set $S_{\cal B}^k$ the subsystem of $S_{\cal B}$ that obtained from the products in $\cal W_k.$ Note that any $k$ subsystem $S_{\cal B}^k$ has solution because of finite dimension nilpotent Lie algebra $\cal W_k$. Since any finite subsystem of the system $S_{\cal B}$ is also subsystem of $S_{\cal B}^k$ for some $k\in \mathbb{N},$ it has a solution. Now applying result of \cite{Abian} we conclude that the system $S_{\cal B}$ has a solution, it leads that $d_0^x$ is a derivation of $\cal N$.

Thanks to Theorem \ref{prop1.20} we obtain linearly independent $s$ diagonal derivations $d_0^{x_i}$ that are embedded into a maximal torus. Thus, $s \leq rank(\cal N)$.
\end{proof}

\begin{cor} \label{cor1.3} The diagonals of $ad_{{x_i}|\mathcal{N}}, \ 1\leq i \leq s,$
forms a basis of a maximal torus $\cal T$ on $\cal N$.
\end{cor}

\subsection{Maximal residually solvable extensions of pro-nilpotent Lie algebras of maximal rank.}

In this subsection, we prove the innerness of all derivations of maximal residually solvable extensions of pro-nilpotent Lie algebras of maximal rank. First, similarly to finite dimensional nilpotent Lie algebras case (see in \cite{Complete1}), we introduce the notions of pro-nilpotent and residually solvable Lie algebras of maximal rank.

\begin{defn} A pro-nilpotent Lie algebra $\cal N$ is called of maximal rank, if $\dim (\cal N/ \cal N^2)=rank(\cal N)$.
\end{defn}

Consider the $\mathbb{N}$-graded infinite-dimensional Lie algebra $\mathfrak{n}_1$, which is the positive part of the affine Kac-Moody algebra $A_1^{(1)}$ (see in \cite{Kac}):
$$[e_i, e_j] =c_{i,j} e_{i+j}, \quad i,j\in \mathbb{N} \quad \mbox{where} \quad c_{i,j}=
\left \{\begin{array}{cccccccc}
1, & if & i-j&\equiv& 1 & mod \ 3,\\[1mm]
0,&if & i-j&\equiv& 0 & mod\ 3,\\[1mm]
-1,& if& i-j&\equiv& -1 & mod\ 3.
\end{array}\right.$$

From the description of $\Der(\mathbf{n}_1)$ obtained in \cite{Residual2} along with
$$\mathbf{n}_1^1/\mathbf{n}_1^2=\mathrm{Span} \{e_1,e_2\} \quad \mathbf{n}_1^i/\mathbf{n}_1^{i+1}=\mathrm{Span} \{e_{i+1}\}, \quad i\geq 2,$$
we conclude that the algebra $\mathbf{n}_1$ is pro-nilpotent Lie algebra of maximal rank.

Furthermore, unlike in the finite-dimensional case, the characterization of the maximal residually solvable extensions of $\mathbf{n}_1$ provided in \cite{Residual2} shows that some of these extensions admit complementary subspace $\mathcal{Q}$ to $\mathbf{n}_1$, which can not be transformed into abelian subalgebra.

\begin{thm} \label{thm5.2} Let $\mathcal{R}$ be a maximal residually solvable extension of a pro-nilpotent Lie algebra $\mathcal{N}$ of maximal rank whose complementary subspace in $\mathcal{R}$ is abelian subalgebra. Then $\mathcal{R}$ is centerless and $\operatorname{Der}(\mathcal{R}) = \operatorname{InnDer}(\mathcal{R})$.
\end{thm}
\begin{proof} Let $\cal R=\cal N\oplus \cal Q$ with $\cal N=\operatorname{Span}\{e_i \ | \ i\in \mathbb{N}\}$ and  $\cal Q=\operatorname{Span}\{x_1, \dots, x_s\}$. Assume that $\{e_1, \dots, e_{s}, x_1, \dots, x_s\}$ are generators of $\cal R$. The condition $\cal R^2 \subseteq \cal N$ implies that $\operatorname{ad}_{x_i}(\cal N^k)\subseteq \cal N^k.$ Therefore for any $x_i\in \cal Q$ one can consider reduced action $\operatorname{ad}_{x_i}$ on $\cal N_k=\cal N/ \cal N^k.$ Denoting by $\cal Q_k$ the vector space $\cal Q$ with the reduced action on $\cal N_k$, for any $k \in \mathbb{N}$ the quotient algebra
$$\cal R_k:= \cal R/ \cal N^k= \cal N/ \cal N^k \oplus \cal Q_k$$
is maximal solvable extension of finite-dimensional nilpotent Lie algebra of maximal rank.

Observe that $\cal N=[\cal N, \cal Q].$ Indeed, for an arbitrary $y\in \cal N$ there exists $k_1$ such that $y\in \cal N \setminus \cal N^{k_1}.$ Since the algebra $\cal R_{k_1}$ is isomorphic to $\cal N_{k_1}\rtimes \cal T_{k_1},$ where $\cal T_{k_1}$ is a maximal torus on $\cal N_{k_1}$, then $\cal N_{k_1}=[\cal N_{k_1}, \cal T_{k_1}]$. Consequently, we get $\cal N_{k_1}=[\cal N_{k_1}, \cal Q_{k_1}]$, which leads the existence of $x\in \cal Q_{k_1}$ and $z\in \cal N_{k_1}$  such that $y=[z,x].$ Let $z\in \cal N \setminus \cal N^{k_2}$. Then we have the product $y=[z,x]$ in the space $\cal N \setminus \cal N^{k}$ for any $k\geq \operatorname{max}\{k_1, k_2\}.$ Thus, $\cal N=[\cal N, \cal Q].$

Let $d$ be an arbitrary derivation of $\cal R$. Then from
$$d(\cal N)=d([\cal N,\cal Q])=
[d(\cal N), \cal Q]+[\cal N, d(\cal Q)]\subseteq \cal R^2 \subseteq \cal N,$$
we deduce that $d(\cal N)\subseteq \cal N,$ which leads to $d(\cal N^k)\subseteq \cal N^k$ for any $k\in \mathbb{N}.$ Therefore, for $d \in \text{Der}(\cal R)$ one can consider the well-defined derivation of $\cal R_k$ as follows $\bar{d}(\bar{v}) = \overline{d(v)},$ where $\bar{v} = v + \cal N^k$. Since $d$ is completely determined by its values on generators it is sufficient to prove the existence of $a \in \cal R$ such that $d(z) = \text{ad}_a(z)$, where $z\in \{e_1, \dots, e_{s}, x_1, \dots, x_s\}$. It is known that $\Der(\cal R_k)=\operatorname{InDer}(\cal R_k)$ for any $k\in \mathbb{N}$ (see \cite{Complete1}). We set $$d(e_i) = \sum\limits_{m=1}^p a_{i,m} e_m, \quad d(x_i) = \sum\limits_{m=1}^p b_{i,m}e_m+ \sum\limits_{m=1}^sc_{i,m}x_m, \quad 1\leq i \leq s.
$$
Let us take sufficiently large $k$ such that $q_k \geq p.$ Then $\overline{d(v)}=\overline{d(\bar{v})}=\operatorname{ad}_{\bar{y}_k}$ for some $\bar{y}_k=y_k+\mathcal{N}^k$ and for any $\bar{v}=v+\mathcal{N}^k$. Set
$$\bar{y}_k=\sum_{m=1}^{q_k} \alpha_{k,m} e_m+\sum_{m=1}^s \beta_{k,m}x_m+\mathcal{N}^k, \quad \mbox{where} \quad \cal N^k=\operatorname{Span}\{e_{q_{k}+i} \ | \ i\in \mathbb{N}\}.$$

Consider equalities
\begin{equation}\label{eq5.2}
\overline{d\left(x_i\right)}=\left[\bar{x}_i, \bar{y}_k\right], \quad \overline{d\left(e_i\right)}=\left[\bar{e}_i, \bar{y}_k\right], \quad 1 \leq i \leq s.
\end{equation}

From the first equality of \eqref{eq5.2} we get
$$\begin{gathered}
\sum_{m=1}^p b_{i, m} e_m+\sum_{m=1}^s c_{i, m} x_m-\sum_{m=1}^{q_k} \alpha_{k,m}\left[x_i, e_m\right]-\sum_{m=1}^s \beta_{k,m}[x_i, x_m]\in \mathcal{N}^k, \quad 1\leq i \leq s.
\end{gathered}$$

Then applying the condition $[\cal Q, \cal Q]=0$ in the last equality, we deduce

$$\begin{gathered}
\sum_{m=1}^p b_{i, m} e_m+\sum_{m=1}^s c_{i, m} x_m-\sum_{m=1}^{q_k} \alpha_{k,m}\left[x_i, e_m\right]\in \mathcal{N}^k, \quad 1\leq i \leq s.
\end{gathered}$$

Thanks to Theorem \ref{prop1.20}, we can assume $[x_i, e_j]=-\sum\limits_{m=j}^{q_k} \lambda_{j,m}^i e_m +\cal N^k.$ Therefore, from the above we derive that $d(x_i)\in \cal N$ and
\begin{equation} \label{eq5.3}
\sum_{m=1}^p \Big( b_{i, m} +\sum\limits_{j=1}^{m}\alpha_{k,j}\lambda_{j,m}^i\Big)e_m-\sum_{m=p+1}^{q_k}\Big(\sum\limits_{j=1}^{m} \alpha_{k,j}\lambda_{j,m}^i\Big)e_m=0.
\end{equation}

Now applying Corollary \ref{cor1.3} we derive that for any $m$ there exists $i$ such that $\lambda_{m,m}^i\neq 0.$ Therefore, from \eqref{eq5.3} we conclude that for any $m$ the coefficients $\alpha_{k,m}$ do not depend on $k$. Similarly from the second equality of \eqref{eq5.2} we derive that coefficients $\beta_{k,m}$ are also independent of $k$. Consequently, all coefficients $\alpha_{k,m}$ and $\beta_{k,m}$ appearing in the decomposition of $y_k$ are independent of $k$. This implies that $\bar{y}_k- \bar{y}_{k+1}\in \mathcal{N}^k$ for all $k\in\mathbb{N}$.

Since representation of $\bar{y}_k$ in the Hamel basis is independent of $k$, we can conclude $y:= y_k$ is well-defined for any $k$. Defining the subspaces  $\mathcal{V}_k = \operatorname{Span}\{\mathcal{N} \setminus \mathcal{N}^k, x_1, \dots, x_s\}$, we find that
$d(v)_{| \mathcal{V}_k} = {\operatorname{ad}_y(v)}_{| \mathcal{V}_k}$
holds for all $v \in \{e_i, x_i \mid 1 \leq i \leq s\}$. Finally, since  $\bigcup\limits_{k=1}^{\infty} \mathcal{V}_k = \mathcal{R}$, it follows that $d = \operatorname{ad}_y$.

For an arbitrary $z\in \operatorname{Center}(\cal R)$, we have  $\bar{z}_k\in \operatorname{Center}(\cal R_k)$ for any $k$. Due to \cite{Superall} the finite-dimensional Lie algebra $\cal R_k$ is isomorphic to $\cal N_k\rtimes \cal T_k$, which is centerless (see in \cite{Complete1}). Therefore, we conclude that $\bar{z}_k=\bar{0}$ for any $k$. If $z\neq 0$, then since $z$ in Hamel basis, for arbitrary large $k$ we can say  $z\notin  \cal V_k$. Thus, we get $z=0$.
\end{proof}

\begin{cor} Let $\cal R=\cal N\rtimes \cal T$ be a maximal solvable extension of pro-nilpotent Lie algebra $\cal N$ of maximal rank. Then $\operatorname{Center}(\cal R)=0$ and $\Der(\cal R)=\operatorname{InDer}(\cal R).$
\end{cor}

Note that in the proof of Theorem \ref{thm5.2} we have used the property $d(\cal N) \subseteq \cal N$ for any $d\in \operatorname{Der}(\cal R).$ The next example shows that this result, which is true in finite-dimensional case, is not true for residually solvable extensions of pro-nilpotent Lie algebras, in general.
\begin{exam} \cite{Residual1} Consider a Lie algebra $\mathcal{R}$  with a basis $\{e_i, x \ | \ i\in \mathbb{N}\}$ whose multiplications table is defined on basis elements by:
$$[e_1,e_i]=e_{i+1}, \quad i\geq2, \quad [x, e_i]=e_i, \quad i\geq2.$$
Then $\cal R$ is residually solvable Lie algebra with pro-nilpotent radical $\rm m_1$ and the map $d$, defined by
$$d(e_1)=x\quad d(e_i)=(i-2)e_{i-1},\ i\geq 2, \quad d(x)=0,$$
is a derivation of $\cal R$ such that $d(\rm m_1)\nsubseteq  \rm m_1.$
\end{exam}

\section{More examples of pro-nilpotent and pro-solvable Lie algebras.}

For a given pro-nilpotent Lie algebra $\cal N$ and associative-commutative algebra $\cal A$, we consider current algebra $\widetilde{\cal N}=\cal N \otimes \cal A$ with the  product:
$$[x \otimes a, y \otimes b]= [x,y] \otimes ab, \quad x,y \in \cal N, a,b \in \cal A.$$

Then for this Lie algebra, we have $\bigcap\limits_{i=1} ^{\infty}\widetilde{\cal N}^{i}=
\bigcap\limits_{i=1} ^{\infty}\cal N^{i}\otimes \bigcap\limits_{i=1} ^{\infty}\cal A^{i}=0.$ Therefore, the current algebra $\widetilde{\cal N}$, constructed from a residually nilpotent Lie algebra  $\cal N$, is itself residually nilpotent.

Moreover, if $\dim (\cal A^i /\cal A^{i+1}) < \infty$ for infinite-dimensional algebra $\cal A$, then
$$\dim (\widetilde{\cal N}^i /\widetilde{\cal N}^{i+1})=
\dim \big(\cal N^i/ \cal N^{i+1}\big) \cdot \dim \big(\cal A^{i}/ \cal A^{i+1}\big) < \infty.$$
It implies that $\widetilde{\cal N}$ is pro-nilpotent Lie algebra.

\begin{exam} For a given pro-nilpotent Lie algebra $\cal N$ we consider subalgebra of the current algebra $\widetilde{\cal N}=\cal N\otimes t\mathbb{C}[t]$ with the product
$[x\otimes t^n, y\otimes t^m]=[x,y]\otimes t^{n+m}.$ Then $\cal N$ is pro-nilpotent Lie algebra.
\end{exam}

It should be noted that similar constructions can be considered for residually solvable (pro-solvable) Lie algebras. Namely, the tensor product $\widetilde{\cal R}=\cal R \otimes \cal A$ with a pro-solvable Lie algebra $\cal R$ and with $\dim (\cal A^i /\cal A^{i+1}) < \infty$ is pro-solvable algebra.

For given pro-nilpotent Lie algebras $\cal N_1$ and $\cal N_2$ for the direct sum $\cal N=\cal N_1 \oplus \cal N_2$ (with component-wise product) we have
$$\bigcap\limits_{i=1} ^{\infty}\cal N^{i}=
\bigcap\limits_{i=1} ^{\infty}\cal N_1^{i}\oplus \bigcap\limits_{i=1} ^{\infty}\cal N_2^{i}=0, \quad \dim (\cal N^i /\cal N^{i+1})=
\dim \big(\cal N_1^i/ \cal N_1^{i+1}\big) + \dim \big(\cal N_2^{i} / \cal N_2^{i+1}\big) < \infty,$$
from we conclude that the direct sum of finitely many pro-nilpotent Lie algebras is pro-nilpotent Lie algebra.

\begin{exam}
Let $\mathcal{N}$ be a pro-nilpotent Lie algebra. Consider the current algebra
\[
\widetilde{\mathcal{N}}=\mathcal{N}\otimes t\mathbb{C}[t]
\]
with the Lie product defined on homogeneous tensors by
\[
[x\otimes t^n,\, y\otimes t^m]=[x,y]\otimes t^{n+m}, \qquad x,y\in \mathcal{N}, \ n,m\in\mathbb{N}.
\]
Then $\widetilde{\mathcal{N}}$ is again a pro-nilpotent Lie algebra.
\end{exam}

Let $(\mathcal{A},\cdot_1)$ and $(\mathcal{B},\cdot_2)$ be associative commutative algebras equipped with pairs of commuting derivations $d_1,d_2\in \operatorname{Der}(\mathcal{A}), \ \bar d_1,\bar d_2\in \operatorname{Der}(\mathcal{B}), \ [d_1,d_2]=[\bar d_1,\bar d_2]=0.$
Define Lie brackets on $\mathcal{A}$ and $\mathcal{B}$ by
$$[a_1,a_2]_1 = d_1(a_1)d_2(a_2) - d_1(a_2)d_2(a_1), \quad [b_1,b_2]_2 = \bar d_1(b_1)\bar d_2(b_2) - \bar d_1(b_2)\bar d_2(b_1).$$
Then both $(\mathcal{A},[\cdot,\cdot]_1)$ and $(\mathcal{B},[\cdot,\cdot]_2)$ are Lie algebras.

Moreover, if both $(\mathcal{A},[\cdot,\cdot]_1)$ and $(\mathcal{B},[\cdot,\cdot]_2)$ are residually nilpotent (respectively, pro-nilpotent), then $\mathcal{A}\otimes\mathcal{B}$ is residually nilpotent (respectively, pro-nilpotent) as well.

\begin{exam} Consider the algebras $\cal A=x_1^{r_1}x_2^{r_2}\mathbb{C}[x_1, x_2]$ and $\cal B=y_1^{q_1}y_2^{q_2}\mathbb{C}[y_1, y_2]$, where $r_i, q_i\in \mathbb{N}$, equipped with partial derivations. Then $\cal A\otimes \cal B$ is pro-nilpotent Lie algebra.
\end{exam}

\textit{Central extensions of pro-nilpotent Lie algebras.}

Let $\cal N$ be a pro-nilpotent Lie algebra over $\mathbb{F}$ and $\cal V$ a vector space over $\mathbb{F}$. Then the bilinear maps
$\theta : \ \cal N \wedge \cal N \to \cal V$  with
$$\theta([x, y], z) + \theta([z, x], y) + \theta([y, z], x) = 0 \quad \text{for all } x, y, z \in \cal N$$
are called {\it $2$-cocycles}. The set of all $2$-cocycles is denoted $Z^2(\cal N, \cal V)$.

For a given $\theta \in Z^2(\cal N, \cal V)$ we set $\cal N_{\theta} = \cal N \oplus \cal V$ and define a bracket $[ - , - ]$ on $\cal N_{\theta}$ as follows
$$[x + v, y + w] = [x, y]_{\cal N} + \theta(x, y),
$$
where $[ - , - ]_{\cal N}$ is the bracket on $\cal N$. Then $\cal N_{\theta}$ becomes a Lie algebra, called the {\it central extension of $\cal N$ by $\theta$}. It is clear that if dimension of $\cal V$ is finite, then
the algebra $\cal N_{\theta}$ is pro-nilpotent Lie algebra, otherwise, $\cal N_{\theta}$ is residually nilpotent.

The short discussions below on central extensions of finite-dimensional nilpotent Lie algebras due to $dim \cal V < \infty $ are still valid for pro-nilpotent Lie algebra $\cal N$ (even it is centerless).

For a linear map $v :  \ \cal N \to \cal V$ we define $\eta(x, y) = v([x, y])$. Then $\eta$ is a $2$-cocycle, called a {\it coboundary}. The set of all $2$-coboundaries is denoted by $B^2(\cal N, \cal V)$. For a $2$-coboundary $\eta$ we have $\cal N_{\theta} \cong \cal N_{\theta + \eta}$.
Consider the set
$$H^2(L, \cal V) = Z^2(L, \cal V)/B^2(L, \cal V).$$

For a Lie algebra $\cal M$ with $Center(\cal M)\neq 0$, we set $\cal V = Center(\cal M)$, and $\cal N = \cal M/Center(\cal M)$.

When constructing nilpotent Lie algebras as $\cal N_{\theta} = \cal N \oplus \cal V$, we want to restrict $\theta$ such that $Center(\cal N_{\theta}) = \cal V$.
For $\theta \in Z^2(\cal N, \cal V)$ we set
$\theta^\perp = \{x \in \cal N \mid \theta(x, y) = 0 \text{ for all } y \in \cal N\}.$
Then $Center(\cal N_{\theta}) = (\theta^\perp \cap Center(\cal N)) + \cal V$.
Similarly to finite-dimensional case for $\theta \in Z^2(\cal N, \cal V)$, we have that $\theta^\perp \cap Center(\cal N) = 0 $ if and only if $Center(\cal N_{\theta}) = \cal V$.

Let $\{e_1, \ldots, e_s\}$ be a basis of $\cal V$, and $\theta \in Z^2(\cal N, \cal V)$. Then $\theta(x, y) = \sum\limits_{i=1}^s \theta_i(x, y)e_i,$
where $\theta_i \in Z^2(\cal N, \mathbb{F})$. Furthermore, $\theta$ is a coboundary if and only if all $\theta_i$ are.

For more detailed description of central extensions of Lie algebras we refer reader to \cite{Graaf}.

\end{document}